\documentclass[review]{elsarticle} 

\usepackage{float}
\usepackage{amsmath,amssymb,amsfonts,amstext,amsmath}
\usepackage{epsfig}
\usepackage{setspace}
\usepackage{hhline}
\usepackage{color,soul}
\usepackage[dvipsnames]{xcolor}
\definecolor{Light}{gray}{.90}
\usepackage[justification=centering]{caption}
\usepackage{multirow}
\usepackage{booktabs}
\usepackage{bm}
\usepackage[english]{babel}
\usepackage{tikz}
\usepackage{tikzscale}
\usepackage{epstopdf}
\usepackage{adjustbox}
\usepackage{longtable}
\usepackage{accents}
\usepackage{lineno}
\usepackage{multirow}
\usepackage{colortbl}
\usepackage[linesnumbered,ruled,vlined]{algorithm2e}
\usepackage{ragged2e}
\usepackage{amsmath,amssymb}
\makeatletter
\usepackage{breqn}
\usepackage{amsthm}
\usepackage{pdflscape}
\usepackage{afterpage}
\usepackage{capt-of}
\usepackage{dsfont}

\newcommand{\cye}{\cellcolor{yellow!75}}

\newcolumntype{x}[1]{>{\raggedright\arraybackslash\hspace{0pt}}p{#1}}
\newtheorem{problem}{Problem}
\newtheorem{assumption}{Assumption}

\newtheorem{lemma}{Lemma}
\newtheorem{theorem}{Theorem}
\newtheorem{remark}{Remark}

\newcommand{\iter}[1]{\ensuremath{\langle #1 \rangle}}
\newcommand{\itrpt}[3]{\ensuremath{\bm{#1}^{\iter{#2}}_{#3}}}
\newcommand{\itrval}[3]{\ensuremath{#1^{\iter{#2}}_{#3}}}

\newcommand{\idxpt}[3]{\ensuremath{\bm{#1}^{(#2)}_{#3}}}
\newcommand{\idxval}[3]{\ensuremath{#1^{(#2)}_{#3}}}

\SetCommentSty{mycommfont}
\SetKwInput{KwInput}{Input}        
\SetKwInput{KwOutput}{Output}       

\DeclareMathOperator*{\argmin}{arg\,min}
\DeclareMathOperator*{\argmax}{arg\,max}

\begin{document}

\begin{frontmatter}
\title{SMGO-$\Delta$: Balancing Caution and Reward\\ in Global Optimization  with Black-Box Constraints}

\author{Lorenzo Sabug Jr.\corref{corr}}\ead{lorenzojr.sabug@polimi.it}
\cortext[corr]{Corresponding author. The first author would like to acknowledge the support of the Department of Science and Technology--Science Education Institute (DOST-SEI) of the Philippines for his research.}
\author{Fredy Ruiz}\ead{fredy.ruiz@polimi.it}
\author{Lorenzo Fagiano\corref{ack2}}\ead{lorenzo.fagiano@polimi.it}
\cortext[ack2]{This research has been supported by the Italian Ministry of University and Research (MIUR) under the PRIN 2017 grant n. 201732RS94 "Systems of Tethered Multicopters".}
\address{Dipartimento di Elettronica, Informazione e Bioingegneria, Politecnico di Milano\\ Piazza Leonardo da Vinci 32, 20133 Milano, Italy}

\begin{abstract}
In numerous applications across all science and engineering areas, there are optimization problems where both the objective function and the constraints have no closed-form expression or are too complex to be managed analytically, that they can only be evaluated through experiments. To address such issues, we design a global optimization technique for problems with black-box objective and constraints. Assuming Lipschitz continuity of the cost and constraint functions, a Set Membership framework is adopted to build a surrogate model of the optimization program, that is used for exploitation and exploration routines. The resulting algorithm, named Set Membership Global Optimization With Black-Box Constraints (SMGO-$\Delta$), features one tunable risk parameter, which the user can intuitively adjust to trade-off safety, exploitation, and exploration. The theoretical properties of the algorithm are derived, and the optimization performance is compared with representative techniques from the literature in several benchmarks. Lastly, it is tested and compared with constrained Bayesian optimization in a case study pertaining to model predictive control tuning for a servomechanism with disturbances and plant uncertainties, addressing practically-motivated task-level constraints.
\end{abstract}
    
\end{frontmatter}

\section{Introduction}
Black-box optimization problems arise in many practical scenarios, such as engineering design, management of complex processes, and control system tuning. A common aspect to such problems is that the objective function involves many interacting factors that are difficult to model mathematically in an accurate way. In these contexts, the objective function is evaluated using (possibly time-consuming) experiments. Another relevant case is when accurate mathematical models do exist, but they can not be easily treated analytically because of their high complexity: in these situations, one can usually rely only on simulations. Typically, these problems feature several local minima and they may also be affected by uncertainty/disturbances, so that the value of the cost may change within certain intervals for the same value of the decision variables. These aspects make the use of established first- or second-order nonlinear programming approaches impractical. Rather, global black-box optimization techniques aim to solve these problems using the results of the experiments/simulations to iteratively choose the value of the decision variable to be tested next, attempting to simultaneously learn the objective function and optimize it. Almost all such techniques try to balance exploitation and exploration, in order to ensure fast convergence and coverage of the search space, respectively. Examples of algorithms for black-box optimization in the unconstrained case are \cite{Jones1998,Gutmann2001,KHONG2013,Malherbe2017,GAO2018,Bemporad2020,Sabug2020CDC,Sabug2020,Urquhart2020}.

Many black-box optimization problems involve also the satisfaction of one or more constraints, deriving for example from the fulfillment of norms and standards, the achievement of a minimum performance threshold, or state constraints. In the literature, most black-box optimization approaches, including the references above, do consider constraint functions that are assumed to be known \textit{a priori} and either analytically or numerically tractable, such as upper and lower bounds on each decision variables, convex sets, or nonlinear but known and differentiable constraint functions. However, there are problems where the constraint values are not known \textit{a priori} or too complicated and uncertain as well, i.e., they  are also black box, leading to uncertain and possibly disconnected feasible sets. This paper focuses on this problem class. In particular, we consider a framework where sensible values of the objective and constraint functions are returned also when the test point is outside the feasible set, i.e., the black-box constraints can be violated during the optimization process. This is in contrast to non-relaxable problems, in which a violated constraint means a crashed simulation or a failed experiment, and the returned objective and/or constraint function values are undefined or not meaningful. Using the taxonomy of black-box constraints introduced in \cite{Digabel2015}, we deal with \texttt{QRSK} (Quantifiable, Relaxable, Simulation, Known) problems.

\subsection*{Previous work}

The interest on considering black-box constraints in black-box optimization is increasing, mainly in synchronous contexts, where the objective and all constraints are jointly evaluated at each sampling point. Efforts have also been made for different problem setups such as in multi-objective contexts \cite{Chen2012,Martinez-Frutos2016,Garrido-Merchan2019}, and in asynchronous evaluations (where the objective and/or constraints can be independently sampled at different test points) \cite{Garrido-Merchan2019,Gelbart2014,Ariafar2019}. However, such contexts are outside the scope of the present work, as we focus on scalar-objective, synchronous constrained optimization.

Swarm- or population-based approaches were formulated in previous works to tackle the problem we treat. For example, \cite{Garg2019} proposes the use of a penalty function, recasting the constrained problem to an unconstrained one, allowing the use of a hybrid between gravitational search and genetic algorithms. An approach proposed by \cite{Jiao2019} recasts the problem into a bi-objective one with the problem objective and overall constraint violation (defined as an average of the individual constraint violations). These population-based algorithms, however, use the sum (or a linear combination) of constraint violations as measure for infeasibility. This leads to high sensitivity w.r.t. constraint function image sets, i.e. a constraint can dominate others because it has a wider range of values.

Another popular class of algorithms is based on the Mesh Adaptive Direct Search (MADS) \cite{Audet2006,ledigabel2011,audet2021nomad}, used in applications as reported in \cite{audet2014,Alarie2021}. In general, the algorithm is composed of search (exploration) and poll (exploitation) phases, selecting evaluation points from a mesh emanating from the current best point, usually along the coordinate directions. Such a mesh becomes more refined as the iterations increase, until the evaluation budget is exhausted. It features flexibility in terms of the selection criteria for generating evaluation points, especially for the search phase. However, the mesh generation for the search phase is exponential w.r.t. dimensionality. Furthermore, the candidate points selection from a mesh limits the search directions that can be taken for evaluation.



In the last decade, kriging-based methods \cite{Li2017,Boukouvala2014,Shi2019,Dong2021} and Bayesian optimization (BO) \cite{Shahriari2016,Zhan2020} grew in popularity. These techniques have a common base: fitting Gaussian process (GP) priors to existing objective and constraint data, and selecting a point that maximizes an acquisition function over the fitted prior. Earlier works addressed the modification of acquisition function formulations considering the objective and constraints/feasibility \cite{Gelbart2014,Boukouvala2014,Shi2019,Gramacy2011,Gardner2014,Jiang2021}. Some works like \cite{Boukouvala2014,Gardner2014} need feasible point/s in the initial data, but \cite{Gelbart2014} addressed this issue by focusing on feasible region search when no feasible samples exist yet. \cite{Antonio2019} proposed a similar approach, devoting the initial iterations to learn the feasible regions, while \cite{Jiang2021} proposed alternating between exploitation (inside the feasible region) and exploration (discovering the constraint boundaries). Other works \cite{Hernandez-Lobato2015,Hernandez-Lobato2016} leveraged in their sampling point selection the expected information gain w.r.t. existing sampled constrained minimal value. In later works, lookahead-based formulations \cite{Lam2017,Zhang2021} were investigated as well, arguing that one-step (or myopic/greedy) algorithms only consider immediate improvement of solution quality, and tend to ignore long-term information gains. A major limitation of such kriging- and Bayesian-based methods lies in the computationally-intensive GP fitting over the existing data, which only allows their use in low-dimensional problems (most previous literature demonstrated its use up to 5-6 dimensions only). Furthermore, optimizing the acquisition function over these ``cheap'' surrogate surfaces is another problem. Most techniques delegate the surrogate optimization to an external solver, which affects reproducibility of results. Furthermore, while they show empirically good performance, convergence properties are not investigated for the above-discussed methods.

A recently-developed class of constrained algorithms, as in \cite{Alimo2020,Beyhaghi2017,Alimo2021}, utilizes a Delaunay triangulation to quantify the uncertainty throughout the search space, given the existing sampled points. These methods feature flexibility in using any smooth interpolation for the surrogate modeling of the objective and constraints. In addition, their convergence is proven assuming twice-differentiable objective and constraint functions. However, the heavy computational burden of calculating/updating a Delaunay triangulation practically limits their use to low-dimension problems.

In summary, while existing methods of constrained black-box optimization have shown good performance, computational burden issues are still pertinent. This limits the practical use of black-box optimization, especially in higher-dimensional problems and in embedded hardware with limited computational facilities. The reproducibility of results -- a mostly-ignored aspect in previous literature -- is also highly relevant in industrial contexts, in which explainable and traceable history of experiments is desirable. A organized summary comparing the previous work is shown in Table~\ref{table:compare-prev-work}.

\begin{table*}[!t]
\small
    \begin{center}
    	\begin{tabular}{|p{0.85in}|p{2in}|p{2in}|}
    		\hline
    		\textbf{Type} & \textbf{Advantages} & \textbf{Disadvantages}\\ \hline
    		Swarm-based \cite{Garg2019,Jiao2019} & Simple implementation and concept of penalty function; computationally efficient; repeatable & Sensitive to value ranges of respective constraints; convergence not investigated \\ \hline
    		MADS \cite{Audet2006,ledigabel2011,audet2021nomad} & Simple design; high flexibility in global search criterions & Lack of search directions due to meshing, mesh generation is exponential w.r.t. dimensionality \\ \hline
    		Kriging- \cite{Li2017,Boukouvala2014,Shi2019,Dong2021} or BO-based \cite{Gelbart2014,Shahriari2016,Zhan2020,Gramacy2011,Gardner2014} & Iteration-efficient; wide implementation & Computationally heavy; practically limited to low dimensions; convergence not investigated; not completely repeatable \\ \hline
    		Delaunay-based \cite{Alimo2020,Beyhaghi2017,Alimo2021} & Iteration-efficient; flexible to different regression techniques; provably convergent & Computationally heavy; practically limited to low dimensions \\ \hline
    		SMGO-$\Delta$ (this work) & Iteration- and computationally efficient; simple implementation; provably convergent; repeatable; user-tunable exploration behavior & Not memory-efficient; not effective with very small feasible regions \\ \hline
    	\end{tabular}
        \caption{Comparison of representative methods for global optimization \\ with black-box constraints}
    	\label{table:compare-prev-work}
    \end{center}
\end{table*}

\subsection*{Contributions of this work}

The present paper proposes a new approach, the Set Membership Global Optimization With Black-Box Constraints (SMGO-$\Delta$), which is computationally efficient, provably convergent, reproducible, and competitive in iteration-based performance. Deepening the research line of our recent work \cite{Sabug2020CDC, Sabug2020} which only considered unconstrained cases, we use the Set Membership framework \cite{Milanese2004} to build estimated models of both the cost and the constraints. These approximations are then used to predict the fulfillment of the black-box constraints in the unsampled areas. Our contributions are as follows:

\begin{itemize}
    \item We present a computationally efficient estimation model for the objective and constraints, given existing data. The Set Membership (SM)-based model is geometrically based on hypercones, which quantify the uncertainty in a more tractable way than GP-based fitting, and Delaunay partitioning-based methods.
    \item The algorithm introduces exploitation and exploration strategies and sampling point selection based on methodically-generated candidate points, which eliminates the need for external techniques such as swarm optimizers to optimize over the SM response surface. The candidate points generation method allows us to consider a generalized convex search space, in contrast to that of \cite{Sabug2020CDC,Sabug2020} which only applies to convex polytopes. The points generation proposed in SMGO-$\Delta$ offers richer variety of search directions than MADS \cite{Audet2006,ledigabel2011,audet2021nomad}. Furthermore, it allows the algorithm to provide reproducible results, i.e. the same starting sample/s will lead to the same best result, assuming no noise.
    \item We introduce a custom-tunable parameter to change the exploration behavior, ranging from “cautious” (staying in the currently discovered feasible region) to “risky” (allowing discovery of disjoint feasible regions, which might contain the global minimizer). Furthermore, we do not require an initial feasible point, and we automatically prioritize regions where more constraints are estimated as fulfilled, encouraging the search for the feasible region.
    \item The theoretical convergence of the proposed algorithm -- an aspect that is often not rigorously considered in the literature -- is investigated and proven in this paper. The convergence guarantees we give are only subject to Lipschitz continuity assumption for the objective and constraints, milder than in \cite{Alimo2020,Beyhaghi2017,Alimo2021}.
    \item We treat practical implementation concerns w.r.t. bounded noise and search directions, and offer methods to deal with such. Furthermore, the computational complexity of the method is analyzed, and with an iterative implementation, we can improve the complexity to $\mathcal{O}(Dn+n^2)$, where $D$ is the search space dimensionality, and $n$ is the number of existing samples.
    \item The presented SMGO-$\Delta$ is compared with other optimizers in a set of benchmark functions, and is shown to be competitive w.r.t. the state of the art. Furthermore, we compare it side-by-side with constrained Bayesian optimization (CBO) in an engineering problem of black-box tuning of a model predictive controller (MPC) for a servomotor. We show via statistical tests that our proposed method is competitive with CBO in iteration-wise optimization performance, but is much faster in terms of computational time.
\end{itemize}

The first version of the SMGO-$\Delta$ algorithm has been presented in \cite{Sabug2021CDC}, only dealing with noiseless evaluations. The additional original contributions in the present work are: an improvement of the exploitation and exploration mechanisms; the derivation of theoretical guarantees; a discussion on practical implementation issues, such as alternative generation methods for the search directions, and computational complexity; the practical consideration of unknown but bounded additive disturbances for both objective and constraints, which are automatically estimated using existing data; an analysis on the hyperparameter sensitivity; finally the benchmark tests and comparative study on an engineering problem with CBO.

This paper is organized as follows. Section~\ref{sec:prob-state} states the assumptions and the problem we aim to address. The proposed method is introduced in Section~\ref{sec:smgo-d-algo}, followed by the theoretical properties in Section~\ref{sec:algo-analysis} and a discussion on implementation issues in Section~\ref{sec:implement-notes}. The performance and sensitivity analysis with an analytic function is presented in Section~\ref{sec:perf-test}. Benchmarks and comparison with state-of-the-art techniques are given in Section~\ref{sec:syn-tests}. The results of the MPC tuning case study are discussed in Section~\ref{sec:expt-tests}, and we draw the conclusions in Section~\ref{sec:conclusion}.

\section{Problem Statement}
\label{sec:prob-state}

We consider the minimization of a cost function $f(\bm{x}), ~f(\bm{x}): \mathcal{X} \rightarrow \mathbb{R}$, where $\mathcal{X} \subset \mathbb{R}^D$ is a compact and convex search set (if only black-box constraints are present, $\mathcal{X}$ can be taken for example as a large-enough bounding hypercube containing all values of $\bm{x}$ that are meaningful according to the application at hand). This minimization is subject to the satisfaction of constraints $g_s, s=1,\ldots,S$: a constraint $g_s$ is satisfied at point $\bm{x}$ when $g_s(\bm{x}) \geq 0$.

Neither $f$ nor any $g_s$ are assumed to be known. The only \textit{a priori} assumption about $f$ and all $g_s$ is given as follows:

\begin{assumption}
\label{ass:lipschitz}
$f$ and $g_s, s=1,\ldots,S$ are Lipschitz continuous functions over $\mathcal{X}$ with unknown Lipschitz constants $\gamma,\,\rho_1,\ldots,\rho_S$:

\[ 
\small
\begin{array}{rcl}
     f &\in& \mathcal{F}(\gamma)\\
     g_1&\in& \mathcal{F}(\rho_1)\\
     &\vdots&\\
     g_S&\in& \mathcal{F}(\rho_S)
\end{array}
\]

\noindent where
\vspace{-0.25in}

\begin{equation}
\small
    \mathcal{F}(\eta) \doteq \Big\{ h: |h(\bm{x}_1) - h(\bm{x}_2)| ~\leq \eta\|\bm{x}_1 - \bm{x}_2\|, \forall\bm{x}_1,\bm{x}_2 \in \mathcal{X} \Big\}.
\end{equation}
\end{assumption}

\noindent In this paper, $\|\cdot\|$ is the 2-norm (or the Euclidean norm). We also assume to be able to evaluate these functions at a given point $\idxpt{x}{n}{}$, where $n\in\mathbb{N}$ is a counter of sampled data points. We denote the resulting values as $\idxval{z}{n}{}=f(\idxpt{x}{n}{})$ and $\idxval{c}{n}{s}=g_s(\idxpt{x}{n}{})$. This is reasonable in many cases of practical interest, when physical processes have finite sensitivity (Assumption~\ref{ass:lipschitz}), the process is repeatable, and accurate sensors are used. In Section~\ref{sec:implement-notes}, we consider an extension to the case of unknown but bounded additive uncertainty.

For compactness of notation, we also introduce the vector of sampled constraint values as $\idxval{\boldsymbol{c}}{n}{}=[\idxval{c}{n}{1},\ldots,\idxval{c}{n}{S}]^\top$, where $^\top$ denotes the matrix transpose operation.
Finally, we assume a non-empty feasible set. Let

\vspace{-0.1in} 
\[
\small
\mathcal{G}_s = \left\{ \bm{x} \in \mathcal{X} ~:~ g_s(\bm{x}) \geq 0\right\}
\vspace{-0.05in}
\]

\noindent be the set of points satisfying the $s$-th constraint.

\begin{assumption}
\label{ass:existence}
Denoting $\mathcal{G} \triangleq \mathcal{X}\cap\left\{\bigcap_{s=1}^{S} \mathcal{G}_s\right\}$, we have 

\[
\small
     \mathcal{G} \neq\varnothing.
\]
\end{assumption}

\noindent Due to Assumptions~\ref{ass:lipschitz} and \ref{ass:existence}, there exists a global minimizer $\bm{x}^*$ such that

\[
    \bm{x}^* \in \big\{ \bm{x} \in \mathcal{G} ~|~ \forall \bm{x}' \in \mathcal{G}, f(\bm{x}') \geq f(\bm{x}) \big\}.
\]

\noindent with the corresponding minimum $z^* = f(\bm{x}^*)$. We denote the data collected by our optimization algorithm as

\begin{equation}\label{eqn:data-set}
\bm{X}^{\iter{n}}= \left\{(\bm{x}^{(1)},z^{(1)},\bm{c}^{(1)}); (\bm{x}^{(2)},z^{(2)},\bm{c}^{(2)}); \ldots; (\bm{x}^{(n)},z^{(n)},\bm{c}^{(n)}) \right\}.
\end{equation} 

\noindent The tuple describing the best (feasible) sample $(\bm{x}^{*\iter{n}}, z^{*\iter{n}}, \bm{c}^{*\iter{n}})$ is

\begin{equation}
\small
\label{eqn:best-sample}
\begin{array}{rccc}
    (\bm{x}^{*\iter{n}}, z^{*\iter{n}}, \bm{c}^{*\iter{n}}) &= &\arg&\min\limits_{(\idxpt{x}{i}{}, \idxval{z}{i}{}, \idxpt{c}{i}{}) \in \bm{X}^{\iter{n}}} \idxval{z}{i}{} \\
    &&&\mathrm{s.t.}  ~\idxpt{c}{i}{} \geq 0
\end{array}
\end{equation}

\noindent which is the feasible sample with lowest objective value. A lexicographic criterion is used to sort out possible multiple feasible points with the same (best) cost. In the remainder, for the sake of notational simplicity, we refer to the best point as $\bm{x}^{*\iter{n}}$. We make no assumptions about the feasibility of the starting point and the existence of the best point $\bm{x}^{*\iter{n}}$ (as defined in \eqref{eqn:best-sample}) at each iteration. Hence, $\bm{x}^{*\iter{n}}$ may not exist for the first iterations of the procedure, until for some $n$ the test point is feasible, after which it will exist for all succeeding iterations. The theoretical properties of SMGO-$\Delta$, which we introduce later on, guarantee that indeed a feasible point will be sampled at finite iterations, under Assumption \ref{ass:existence}.

Now we are ready to state the problem addressed in this paper.

\begin{problem}
Design an algorithm that generates a sequence of points $\left\{ \idxpt{x}{1}{}, \idxpt{x}{2}{}, \ldots \right\}$, $\idxpt{x}{i}{} \in \mathcal{X}$, such that

\vspace{-0.2in}
\begin{align*}
    \forall \varepsilon>0, \exists n_\varepsilon < \infty ~:~ & z^{*\iter{n_\varepsilon}} \leq z^* + \varepsilon, \\
    ~ & c^{*\iter{n_\varepsilon}}_s \geq 0, s=1,\ldots,S.
\end{align*}
\end{problem}

\section{Set Membership Global Optimization With Black-Box Constraints (SMGO-$\Delta$): Algorithm}
\label{sec:smgo-d-algo}


As usual in the literature, the search for $\bm{x}^*$ is carried out by a sequential sampling procedure, wherein the next test point is decided on the basis of the existing data. The inclusion of black-box constraints implies significant changes in the algorithm design w.r.t. the unconstrained case \cite{Sabug2020}, as discussed in the following subsections.

\subsection{Update of the data-set and of the Lipschitz constants and disturbance amplitudes}
At each iteration $n$, a point $\idxpt{x}{n}{}\in\mathcal{X}$ is tested by the algorithm, chosen with a strategy described later on. The corresponding sampled tuple $(\idxpt{x}{n}{}, z^{(n)}, \idxval{\boldsymbol{c}}{n}{})\in\mathbb{R}^{D+1+S}$ is added to the data-set $\itrpt{X}{n}{}$:
\vspace{-0.2in}

\[
    \itrpt{X}{n}{} = \itrpt{X}{n-1}{} \cup (\idxpt{x}{n}{}, z^{(n)}, \idxval{\boldsymbol{c}}{n}{}).
\]

From the updated data-set $\itrpt{X}{n}{}$, the algorithm updates the Lipschitz constants' estimates $\tilde{\gamma}^{\iter{n}}$ and $\tilde{\rho}_s^{\iter{n}},\,s=1,\ldots,S$ (see Assumption \ref{ass:lipschitz}) as (see also \cite{Milanese2004}):
\vspace{-0.125in}

\begin{equation}
\small
    \label{eqn:gamma-update}
    \itrval{\tilde{\gamma}}{n}{} = \max_{(\idxpt{x}{i}{},\idxval{z}{i}{},\idxpt{c}{i}{}), (\idxpt{x}{j}{},\idxval{z}{j}{},\idxpt{c}{j}{})\in\itrpt{X}{n}{}} \left( 
      \frac{|\idxval{z}{i}{}-\idxval{z}{j}{}|}{\|\idxpt{x}{i}{}-\idxpt{x}{j}{}\|}, \underline{\gamma}\right),
\end{equation}

\begin{equation}
\small
    \label{eqn:rho-update}
    \itrval{\tilde{\rho}}{n}{s} = \max_{(\idxpt{x}{i}{},\idxval{z}{i}{},\idxpt{c}{i}{}), (\idxpt{x}{j}{},\idxval{z}{j}{},\idxpt{c}{j}{})\in\itrpt{X}{n}{}} \left(
      \frac{|\idxval{c}{i}{s}-\idxval{c}{j}{s}|}{\|\idxpt{x}{i}{}-\idxpt{x}{j}{}\|}, \underline{\rho}_s \right),
\end{equation}

\noindent where $\underline{\gamma} > 0,\,\underline{\rho}_s > 0$ are small but finite initial estimates for $\gamma,\,\rho_s,\,s=1,\ldots,S$, respectively. 

The above-given estimated Lipschitz constants are, by construction, unfalsified by the available data. At $n=1$, one can set these estimates to $\underline{\gamma}, \underline{\rho}_s$, and select the test point $\idxpt{x}{1}{}$ with a strategy of choice (for example a random starting point, or, as discussed in our simulation example, a sensible point for the application at hand).

We can now build the following upper- and lower-bound functions, $\overline{f}^{\iter{n}}(\bm{x})$ and $\underline{f}^{\iter{n}}(\bm{x})$, resorting to a Set Membership approach \cite{Milanese2004}:
\vspace{-0.05in}



\begin{equation}
\small
  \label{eqn:f-upper-bounds}
  \overline{f}^{\iter{n}}(\bm{x}) \triangleq \min_{k = 1, \ldots, n} \left(z^{(k)} + \tilde{\gamma}^{\iter{n}} \|\bm{x}-\bm{x}^{(k)} \|\right),
\vspace{-0.1in}
\end{equation}

\begin{equation}
\small
  \label{eqn:f-lower-bounds}
  \underline{f}^{\iter{n}}(\bm{x}) \triangleq \max_{k = 1, \ldots, n} \left(z^{(k)} - \tilde{\gamma}^{\iter{n}} \|\bm{x}-\bm{x}^{(k)} \|\right).
\end{equation}

\noindent These functions represent the tightest bounds for the unsampled regions, given the sampled points and the Lipschitz continuity assumption. Furthermore, we build the central approximation of the objective function,

\[
\small
    \tilde{f}^{\iter{n}}(\bm{x}) = \frac{1}{2}\left(\overline{f}^{\iter{n}}(\bm{x}) + \underline{f}^{\iter{n}}(\bm{x})\right)
\]

\noindent and the uncertainty measure

\[
\small
    \itrval{\lambda}{n}{}(\bm{x}) = \overline{f}^{\iter{n}}(\bm{x}) - \underline{f}^{\iter{n}}(\bm{x}).
\]

A visual interpretation of the SM-based bounds and function approximation is shown in Fig.~\ref{fig:sm-model}. The same is performed to calculate the quantities $\overline{g}_s^{\iter{n}}(\bm{x})$, $\underline{g}_s^{\iter{n}}(\bm{x})$, $\tilde{g}_s^{\iter{n}}(\bm{x})$, and $\pi_s^{\iter{n}}(\bm{x})$ for each constraint function $g_s$:


\begin{subequations}\label{eqn:constr-approx}
\begin{gather}
\small
  \overline{g}_s^{\iter{n}}(\bm{x}) \triangleq \min\limits_{k = 1, \ldots, n} \left(c_s^{(k)} + \tilde{\rho}_s^{\iter{n}} \|\bm{x}-\bm{x}^{(k)} \|\right)\\
  \underline{g}_s^{\iter{n}}(\bm{x}) \triangleq \max\limits_{k = 1, \ldots, n} \left(c_s^{(k)} - \tilde{\rho}_s^{\iter{n}} \|\bm{x}-\bm{x}^{(k)} \|\right)\\
    \tilde{g}_s^{\iter{n}}(\bm{x}) = \frac{1}{2}\left(\overline{g}_s^{\iter{n}}(\bm{x}) + \underline{g}_s^{\iter{n}}(\bm{x})\right)\label{eqn:constr-approx-central}\\
    \itrval{\pi}{n}{s}(\bm{x}) = \overline{g}_s^{\iter{n}}(\bm{x}) - \underline{g}_s^{\iter{n}}(\bm{x}).
\end{gather}
\end{subequations}

\begin{figure}[!t]
	\centering
	\includegraphics[width=\columnwidth]{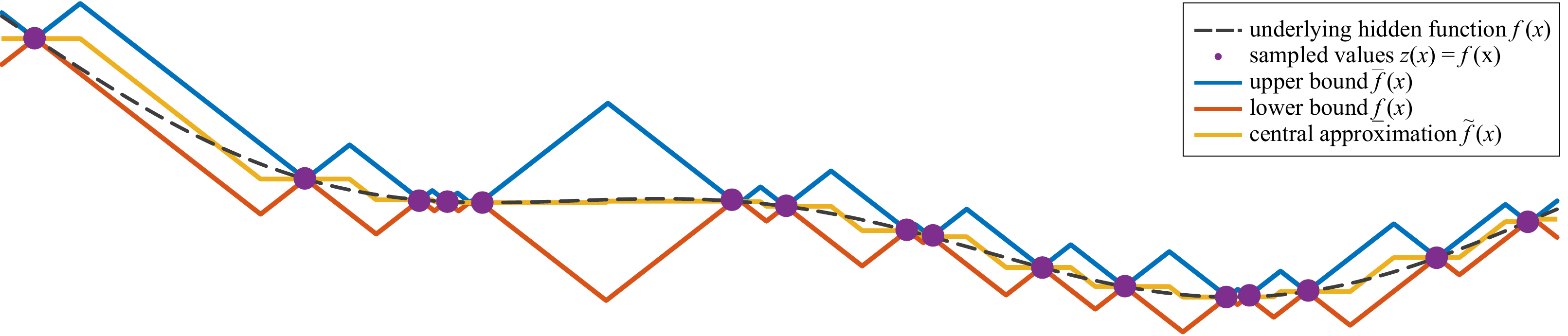}
	\caption{SM-based upper- and lower bounds,\\ and function approximation}
	
	\label{fig:sm-model}
\end{figure}

After updating the current best point and the estimates of Lipschitz constant, upper and lower bounds, functions values and uncertainty, the proposed optimization algorithm attempts first an exploitation strategy, possibly followed (if not satisfactory enough) by an exploration one. These two routines are described next.


\subsection{Exploitation with constraints}\label{subsec:exploit}
This subroutine attempts to improve on the current best value by searching $\mathcal{X}$ for a better candidate point according to the lower bound of the cost function, while satisfying the estimated constraints. Let us denote with $\itrpt{E}{n}{}\in\mathcal{X}$ a finite set of candidate points, selected according to a gridding strategy described in Section \ref{subsec:cdpt-generation}, and with $\mathcal{T}^{\iter{n}}$ a trust region around $\bm{x}^{*\iter{n}}$, discussed in Section \ref{subsec:trust-region}. Let us further introduce the user-chosen scalar $\Delta \in [0, 1]$. Then, the exploitation routine solves the following problem:
\vspace{-0.1in}

\begin{subequations}
\begin{gather}
    \small
    \itrpt{x}{n}{\theta} = \arg\min_{\bm{x} \in \itrpt{E}{n}{} \cap \mathcal{T}^{\iter{n}}} \xi_\theta(\bm{x})  \label{eqn:exploit}\\
    \text{s.t.} ~\Delta\tilde{g}_s^{\iter{n}}(\bm{x}) + (1-\Delta)\underline{g}_s^{\iter{n}}(\bm{x})\geq 0, s = 1, \ldots, S \label{eqn:exploit-constraint}
\end{gather}
\end{subequations}

\noindent where the exploitation cost is defined as 

\begin{equation}
\label{eqn:exploit-cost-nominal}
    \xi_\theta(\bm{x}) \triangleq \tilde{f}^{\iter{n}}(\bm{x}) - \beta\lambda^{\iter{n}}(\bm{x}),
\end{equation}

\noindent and $\beta$ is a user-defined weighting parameter. The objective function in \eqref{eqn:exploit} is a trade-off between minimizing the central approximation of the cost function $f$ ($\beta=0$), and maximizing the associated uncertainty in order to gain more information with the next sample ($\beta>0$). As default value, we set $\beta=0.1$ as in~\cite{Sabug2021CDC}.

We confine the optimization \eqref{eqn:exploit} inside a trust region $\mathcal{T}^{\iter{n}}$ to search in the vicinity of the existing best point $\bm{x}^{*\iter{n}}$, evaluating whether there is a neighboring point that improves on the current best. Furthermore, our design of exploitation cost $\xi_\theta(\bm{x})$ aims to choose a point that improves as much as possible from $\bm{x}^{*\iter{n}}$, by using the lower bounds. Resolving the terms in \eqref{eqn:exploit-cost-nominal}, it is actually a weighted sum of $\overline{f}^{\iter{n}}(\bm{x})$ and $\underline{f}^{\iter{n}}(\bm{x})$ which results in favoring points that are close to $\bm{x}^{*\iter{n}}$, aiming for a local optimization. 

The hyperparameter $\Delta$ in \eqref{eqn:exploit-constraint} is referred to as the \textit{risk factor}: it affects the cautiousness in choosing the exploitation point. $\Delta = 0$ provides the most cautious choice of a feasible candidate point, because it uses the constraints' lower bounds $\underline{g}_s^{\iter{n}}(\bm{x})$ to estimate the feasible set, thereby shrinking it considerably. On the other hand, $\Delta = 1$ is associated with a riskier exploitation behavior, using the central approximation $\tilde{g}_s^{\iter{n}}(\bm{x})$ and correspondingly expanding the estimated feasible set.

If \eqref{eqn:exploit} is estimated to be feasible, the exploitation routine then carries out a test on the expected improvement obtained by the resulting point $\itrpt{x}{n}{\theta}$, w.r.t. the current best value, as done also in \cite{Sabug2020}:

\vspace{-0.1in}

\begin{equation}
\small
\label{eqn:exploit-cond}
    \itrval{\underline{f}}{n}{}(\itrpt{x}{n}{\theta}) \leq z^{*\iter{n}} - \eta,
\end{equation}

\noindent where $\eta = \alpha \itrval{\tilde{\gamma}}{n}{}$ is the expected improvement threshold. If this test is passed, SMGO-$\Delta$ will choose to sample the selected point, i.e. $\idxpt{x}{n+1}{} = \itrpt{x}{n}{\theta}$. On the other hand, if \eqref{eqn:exploit} is estimated infeasible or condition \eqref{eqn:exploit-cond} is not met, the optimization algorithm proceeds to exploration. Through the expected improvement test, SMGO-$\Delta$ avoids being stuck in a local minimum, as will be proven in Section~\ref{sec:algo-analysis}.


\begin{remark}
In this paper, for simplicity we loosely use the term ``feasibility'' and ``constraint satisfaction'' depending on the context. When we describe a sample $\idxpt{x}{i}{} \in \itrpt{X}{n}{}$ as being feasible or having satisfied/fulfilled a constraint $g_s$, we mean that the constraints are actually satisfied as evaluated in the corresponding simulation/experiment, i.e., $\idxval{c}{i}{s} \geq 0$. However, we consider that an unsampled point $\bm{x} \in \mathcal{X} \setminus \itrpt{X}{n}{}$ is feasible or satisfies constraint $g_s$ if it is \underline{predicted} to fulfill the constraint/s based on its central estimate \eqref{eqn:constr-approx-central}, i.e., $\tilde{g}_s(\bm{x}) \geq 0$.
\end{remark}

\subsection{Exploration by uncertainty and estimated constraint violations}
The exploration routine probes the areas of the search space where cost function uncertainty is largest, while at the same time penalizing possible constraint violations. The exploration point $\itrpt{x}{n}{\psi}$ is chosen as:
\vspace{-0.125in}

\begin{align}
\label{eqn:explore}
    \itrpt{x}{n}{\psi} = \arg\max_{\bm{x} \in \itrpt{E}{n}{}} \itrval{\xi}{n}{\psi}(\bm{x})
\end{align}


\noindent with the \textit{exploration merit function} defined as:

\begin{equation}
\label{eqn:explore-merit}
    \itrval{\xi}{n}{\psi}(\bm{x}) \triangleq \itrval{h}{n}{}(\bm{x}) + k(\itrval{\tau}{n}{}(\bm{x})),
\end{equation}

\noindent where $\itrval{\tau}{n}{}(\bm{x})$ describes the age of the candidate point, i.e. the iterations elapsed since its creation. Furthermore, $k(\cdot)$ must be chosen as a class-$\mathcal{K}_\infty$ function, i.e. such that it is continuous, strictly increasing, $\lim_{\tau \rightarrow \infty} k(\tau) = +\infty$, and $k(0) = 0$. The inclusion of $k(\cdot)$ encourages the exploration of the most remote areas in $\mathcal{X}$, and in turn, guarantees the convergence of SMGO-$\Delta$, as proven in Section \ref{sec:algo-analysis}. Finally, $\itrval{h}{n}{}(\cdot)$ is a user-chosen merit function, whose sole requirement is to be bounded, i.e., that $\forall \bm{x} \in \mathcal{X}, n \in [1, +\infty), | \itrval{h}{n}{}(\bm{x}) | < \infty$. In principle, this function shall be chosen to trade-off constraint satisfaction and exploration of points where uncertainty is largest. This is in contrast with our work in the unconstrained case \cite{Sabug2020CDC,Sabug2020} where we only search for points with highest uncertainty. In this work, we propose the following:

\begin{equation}
\label{eqn:exploration-weight}
\itrval{h}{n}{}(\bm{x}) \triangleq d(\bm{x})\Big((1-\Delta)w_\lambda(\bm{x}) + \Delta w_\pi(\bm{x})w_g(\bm{x})\Big),
\end{equation}
where:
\begin{equation}
\label{eqn:explore-remoteness-measure}
    d(\bm{x}) = \min_{(\idxpt{x}{i}{},\idxval{z}{i}{},\idxpt{c}{i}{}) \in \itrpt{X}{n}{}} \| \idxpt{x}{i}{} - \bm{x} \|,
    \end{equation}
\begin{align}
\small
\label{eqn:explore-sum-lambda}
    w_\lambda(\bm{x}) = 
    \begin{cases}
      \lambda^{\iter{n}}(\bm{x}) & \mathrm{if~}\Delta\tilde{g}_s^{\iter{n}}(\bm{x}) + (1-\Delta)\underline{g}_s^{\iter{n}}(\bm{x})\geq 0,\, s = 1, \ldots, S\\
      0 & \text{otherwise,}
    \end{cases}
\end{align}

\begin{align}
\small
\label{eqn:explore-sum-pi}
    w_\pi(\bm{x}) &= \sum_{s=1}^{S}\frac{\pi_s^{\iter{n}}(\bm{x})}{\itrval{\tilde{\rho}}{n}{s}}, \\
\label{eqn:explore-prio}
    w_g(\bm{x}) &= 2^{\sum \mathds{1}_s(\bm{x})}
\end{align}

\begin{equation}
\label{eqn:explore-prio-indicator}
    \mathds{1}_s(\bm{x}) = 
    \begin{cases}
      1 & \text{if}\;\tilde{g}_s^{\iter{n}}(\bm{x}) \geq 0\\
      0 & \text{otherwise.}
    \end{cases}
\end{equation}

The merit function prioritizes more remote candidate points by using $d(\bm{x})$ \eqref{eqn:explore-remoteness-measure} as a multiplier to the other terms in $\itrval{h}{n}{}(\bm{x})$. The term $w_\lambda(\bm{x})$ in \eqref{eqn:exploration-weight} is included to encourage the exploration of points with more uncertainty w.r.t. the cost function $f$, using the same condition as \eqref{eqn:exploit-constraint}, hence exhibiting cautious exploration depending on $\Delta$. This is motivated by our need to discover the shape of the underlying cost $f$ since sampling will shrink the uncertainty $\lambda^{\iter{n}}$ at that point to zero, while correspondingly decreasing the uncertainty at its vicinity. However, we only consider points that we estimate to be feasible, to avoid sampling in unfeasible regions when we already discovered feasible one/s.

The second term in \eqref{eqn:exploration-weight} aims to discover the feasible set, favoring points with higher uncertainty $\itrval{\pi}{n}{s}$ w.r.t. the constraint functions $g_s$. Each estimate $\itrval{\pi}{n}{s}$ is normalized by its Lipschitz constant $\itrval{\tilde{\rho}}{n}{s}$ (see \eqref{eqn:explore-sum-pi}) to be less sensitive with respect to the range of absolute values of each constraint.

We define \eqref{eqn:explore-prio}-\eqref{eqn:explore-prio-indicator} to give priority to points with more (estimated) satisfied constraints, practically doubling the merit with every additional constraint fulfilled at a point. 


The risk factor $\Delta \in [0, 1]$, in a similar manner as in exploitation, affects the level of caution in exploring outside the estimated feasible region. $\Delta \rightarrow 0$ results in higher cautiousness, favoring sampling within the estimated feasible region (we mostly consider $w_\lambda(\bm{x})$ in our exploration).

The selected point $\itrpt{x}{n}{\psi}$ is then directly assigned as the next sample point $\idxpt{x}{n+1}{}$.

\begin{remark}
\label{remark:safe-opt}
The concept of \textit{caution} considered here is different from \textit{safe} optimization as discussed in \cite{Konig2021}. Indeed, because the Lipschitz constants $\rho_s$ of the constraints $g_s$ are unknown, we cannot guarantee \textit{a priori} the safety of a test point. On the other hand, under the additional assumptions of having a feasible starting point and of knowing upper bounds on the Lipschitz constants $\rho_s$, one can use $\Delta=0$ to guarantee robustly hard constraint satisfaction.
\end{remark}

\begin{remark}
\label{remark:age-term}
We have introduced the age-based term in \eqref{eqn:explore-merit}. While guaranteeing convergence, a high slope for $k$ can cause candidate points to be chosen because of their age, and not because of their potential information gain. Thus, $k$ should be chosen with very small slope to mitigate this effect.
\end{remark}

\subsection{Generation of candidate points}\label{subsec:cdpt-generation}
In the selection of the exploitation point $\itrpt{x}{n}{\theta}$ \eqref{eqn:exploit} and the exploration one $\itrpt{x}{n}{\psi}$ \eqref{eqn:explore}, we consider a set of candidates $\itrpt{E}{n}{}$, which is systematically generated based on the existing data. Such approximation eliminates the need for an external optimization algorithm to calculate \eqref{eqn:exploit} and \eqref{eqn:explore}. Furthermore, the proposed algorithm is repeatable, i.e. given the same starting point $\idxpt{x}{1}{}$ and search set $\mathcal{X}$, the algorithm will produce the same sampling points sequence and final result.


In \cite{Sabug2020CDC,Sabug2020}, we proposed a candidate point generation built upon the midpoints among all existing samples, including the vertices of the search set, assumed to be polytopic in those previous works. However, the complexity of such candidate points generation is exponential w.r.t. dimension $D$, in the rather common case of hyperrectangular $\mathcal{X}$. Hence, we propose another candidate points generation strategy with improved complexity w.r.t. $D$.

Given an incoming sample $\idxpt{x}{n}{}$, we cumulatively add candidate points to $\itrpt{E}{n}{}$ according to the following criteria:
\begin{enumerate}
    \item[1)] along the coordinate directions stemming from \idxpt{x}{n}{},
    \item[2)] by gridding between each sampled point and all the other ones.
\end{enumerate}

To generate 1), using the coordinate directions $\pm \hat{\bm{a}}_d, d = 1, \ldots, D$, we define
\vspace{-0.125in}

\begin{align}
\small
    \label{eqn:cdpt-end-d1}
    \idxpt{b}{n}{+d}=& \max_{b \in [0, \infty)} b\\
    \mathrm{s.t.}\ \  & \idxpt{x}{n}{} + b\hat{\bm{a}}_d \in \mathcal{X},\nonumber \\
    \label{eqn:cdpt-end-d2}
    \idxpt{b}{n}{-d} =& \max_{b \in [0, \infty)} b\\
    \mathrm{s.t.}\ \ &\idxpt{x}{n}{} - b\hat{\bm{a}}_d \in \mathcal{X}.\nonumber 
\end{align}


\noindent These are the lengths of the longest segments inside the search set $\mathcal{X}$ departing from $\idxpt{x}{n}{}$ in the coordinate directions. Then, the set of candidate points generated by $\idxpt{x}{n}{}$ along the coordinate directions are
\vspace{-0.125in}

\begin{align}
\small
 \idxpt{Y}{n}{+d} &= \Big\{ \idxpt{x}{n}{} + \frac{k}{B} \idxpt{b}{n}{+d}\hat{\bm{a}}_d, 
 k \in \{ 1, \ldots, B-1 \} \Big\}\\
  \idxpt{Y}{n}{-d} &= \Big\{ \idxpt{x}{n}{} - \frac{k}{B} \idxpt{b}{n}{-d}\hat{\bm{a}}_d, 
 k \in \{ 1, \ldots, B-1 \} \Big\}
\end{align}

\noindent where $1/B$ is the relative grid granularity.

For criterion 2), we define the candidate points along the segments from $\idxpt{x}{n}{}$ pointing to each other sample $\idxpt{x}{j}{}, j=1,\ldots,n-1$ as

\[
    \idxpt{Y}{n}{\circ} = \Big\{ \idxpt{x}{n}{} + \frac{k}{B} \idxpt{x}{j}{}, k \in \{ 1, \ldots, B-1 \}, j \in \{ 1, \ldots, n-1 \} \Big\}.
\]

\noindent Finally, the set of candidate points generated by $\idxpt{x}{n}{}$ is
\vspace{-0.15in}

\begin{equation}
\small
    \idxpt{Y}{n}{} = \left(\bigcup_{d=1}^{D}\left\{ \idxpt{Y}{n}{+d}, \idxpt{Y}{n}{-d} \right\}\right) \bigcup \idxpt{Y}{n}{\circ}.
\end{equation}


\noindent Lastly, the aggregate collection of candidate points at iteration $n$ is now
\vspace{-0.125in}

\[
\small
    \itrpt{E}{n}{} = \itrpt{E}{n-1}{} \cup \idxpt{Y}{n}{} .
\]

\noindent This results in a cumulative total of $n(B-1)(2D+\frac{n-1}{2})$ candidate points, which is polynomial w.r.t. $n$, but linear with $D$, thus computationally advantageous with higher search space dimensions. Furthermore, unlike our proposal in \cite{Sabug2020}, we do not require anymore a polytopic search set $\mathcal{X}$, but only a convex and compact one.

\subsection{Exploitation trust region $\mathcal{T}^{\iter{n}}$}\label{subsec:trust-region}

The set $\mathcal{T}^{\iter{n}}$ in \eqref{eqn:exploit} limits the choice of the exploitation candidate point to a trust region around $\bm{x}^{*\iter{n}}$.

A ball $\mathcal{T}^{\iter{n}}$, centered at $\bm{x}^{*\iter{n}}$ and with a radius $\overline{\upsilon} < 1$, is declared when a feasible best point $\bm{x}^{*\iter{n}}$ is found, 

\begin{equation}
\label{eqn:trust-region}
\mathcal{T}^{\iter{n}}=\left\{\bm{x} \in \mathcal{X}: \|\bm{x}-\bm{x}^{*\iter{n}}\|_\infty \leq \upsilon^{\iter{n}}  \right\}    
\end{equation}

In \eqref{eqn:trust-region}, any norm can be used, e.g. 1- or 2-norm, for example in order to adapt to the characteristics of the search set $\mathcal{X}$. The size of $\mathcal{T}^{\iter{n+1}}$ is updated depending on the \textit{a posteriori} value of the sample $\idxval{z}{n+1}{}$, as follows:

\begin{equation}
\label{eqn:trust-region-update}
    \itrval{\upsilon}{n+1}{} = 
    \begin{cases}
      \min(\underline{\upsilon}, \kappa\itrval{\upsilon}{n}{}) & \text{if exploration was done in }n+1 \text{ or } \idxval{z}{n+1}{} > \idxval{z}{n}{}\\
      \max(\overline{\upsilon}, \frac{1}{\kappa}\itrval{\upsilon}{n}{}) & \text{if exploitation was done in }n+1 \text{ and } \idxval{z}{n+1}{} \leq \idxval{z}{n}{} - \alpha\idxval{\gamma}{n}{}, \\
      ~ & \text{with }\idxval{c}{n+1}{s} \geq 0, s=1,\ldots,S,\\
      \itrval{\upsilon}{n}{} & \text{otherwise.}
    \end{cases}
\end{equation}

\noindent where $\kappa < 1$ is the shrinking factor. In summary, we shrink the trust region when an exploitation fails, i.e. there is no exploitation improvement w.r.t. current best $z^{*\iter{n}}$, or if no exploitation point was found at iteration $n+1$ (exploration was done instead). Conversely, we expand the trust region (up to $\overline{\upsilon}$) when the new sample from exploitation has \textit{a posteriori} returned a feasible point and satisfied the expected improvement threshold.

The minimum size of $\mathcal{T}^{\iter{n}}$ is also limited, that is, if $\itrval{\upsilon}{n}{} \leq \underline{\upsilon}$, we reset the trust region size $\itrval{\upsilon}{n+1}{} = \underline{\upsilon}$. In this paper we set the default trust region-related quantities as $\kappa = 0.5$, $\overline{\upsilon} = 0.1$, and $\underline{\upsilon} = \kappa^{10} \overline{\upsilon}$.

\subsection{Algorithm summary}\label{subsec:summary-implementation}
A pseudo-code of the proposed method is shown in Algorithm~\ref{algo:smgo-d}. 
\begin{algorithm}
    \small
    \DontPrintSemicolon
    \KwInput{Initial point $\bm{x}^{(1)}$, search space $\mathcal{X}$ \\ 
    \hspace{0.425in} Lipschitz constants estimates $\itrval{\tilde{\gamma}}{1}{}=\underline{\gamma}$, $\itrval{\tilde{\rho}}{1}{s}=\underline{\rho}$ \\ 
    \hspace{0.425in} Maximum number of iterations $N$. \\
    \hspace{0.425in} Parameters $\alpha, \beta, \psi, \Delta$}
    \While{iteration $n$ within the budget $N$}
    {
        \tcp{Long function evaluation, data update}
        Evaluate the long functions $f$ and $g_s$ at $\bm{x}^{(n)}$, add the resulting sample $(\bm{x}^{(n)}, z^{(n)}, \bm{c}^{(n)})$ to the set $\itrpt{X}{n}{}$\;
        Update Lipschitz constants $\itrval{\tilde{\gamma}}{n}{}$, $\itrval{\tilde{\rho}}{n}{1}, \ldots, \itrval{\tilde{\rho}}{n}{S}$ according to \eqref{eqn:gamma-update}-\eqref{eqn:rho-update}, the current best sample $(\bm{x}^{*\iter{n}}, z^{*\iter{n}}, \bm{c}^{*\iter{n}})$ from $\itrpt{X}{n}{}$ \eqref{eqn:best-sample}, candidate points $\itrpt{E}{n}{}$, and the trust region $\mathcal{T}^{\iter{n}}$ \eqref{eqn:trust-region}-\eqref{eqn:trust-region-update}\;
        \tcp{Exploitation routine}
        Solve \eqref{eqn:exploit} to choose the candidate exploitation point $\itrpt{x}{n}{\theta}$ in the estimated feasible region\;
        \If{ $\itrpt{x}{n}{\theta}$ exists and expected improvement condition \eqref{eqn:exploit-cond} is met}
        {
            Assign test point for next iteration $\bm{x}^{(n+1)} \leftarrow \itrpt{x}{n}{\theta}$\;
        }
        \Else
        {
            \tcp{Exploration routine}
            Solve \eqref{eqn:explore} to compute $\itrpt{x}{n}{\psi}$ \;
            Assign test point for next iteration $\bm{x}^{(n+1)} \leftarrow \itrpt{x}{n}{\psi}$\;
        }
        Go to next iteration $n \leftarrow n+1$
    }
    Final optimal point and value: return the best sample ($\bm{x}^{*\iter{N}}, z^{*\iter{N}}, \bm{c}^{*\iter{N}}$) from the set $\itrpt{X}{N}{}$
\caption{SMGO-$\Delta$ Algorithm}
\label{algo:smgo-d}
\end{algorithm}

\section{Algorithm Analysis}
\label{sec:algo-analysis}
In this section, we analyze the properties of SMGO-$\Delta$, with regards to its convergence and its computational complexity.

\subsection{Convergence Guarantees}
We provide theoretical guarantees on the convergence of SMGO-$\Delta$ to the best feasible point, by proving its dense points generation behavior.

In the remainder, given $\bm{x} \in \mathbb{R}^D$ and $r > 0$, we denote $\mathcal{B}(\bm{x},r)$ as

\[
    \mathcal{B}(\bm{x},r) \triangleq \left\{ \bm{y} ~\big|~ \| \bm{y} - \bm{x} \| \leq r \right\}.
\]

We now present the first lemma, showing that the exploitation will end in finite iterations, allowing for SMGO-$\Delta$ to escape local minima, and continue with exploration around the search space.

\begin{lemma}
\label{lemma:exploit-will-fail}
    Mode~$\theta$ (exploitation) will end after finite iterations.
\end{lemma}

\begin{proof}
Mode~$\theta$ ends, and the algorithm switches to Mode~$\psi$, when any of the following occurs:

\begin{enumerate}
    \item[1)] No supposedly feasible point has been found so far: $\forall (\idxpt{x}{i}{},\idxval{z}{i}{},\idxpt{c}{i}{})\in \itrpt{X}{n}{}, \exists~ s \in [1,~S] : \idxval{c}{i}{s}<0$,
    \item[2)] A feasible sample has been observed, but no candidate points are inside the estimated feasible region: $\nexists ~\bm{x} \in \itrpt{E}{n}{} : \tilde{g}_s(\bm{x}) \geq 0, s = 1, \ldots, S$,
    \item[3)] The chosen candidate point $\itrpt{x}{n}{\theta}$ fails the expected improvement test \eqref{eqn:exploit-cond}.
\end{enumerate}

\noindent It is enough to show that at least one of these three conditions occurs in finite iterations to prove the Lemma. In particular, in a way similar to \cite{Sabug2020} (Lemma 2), we can show that this applies to the third case. 

At iteration $n$, consider $\mathcal{E}^{\iter{n}} \subset \mathcal{X}$ as the region eligible for an exploitation, i.e.,

\[
  \mathcal{E}^{\iter{n}} \triangleq \left\{ \bm{x} \in \mathcal{X} ~:~ \underline{f}^{\iter{n}}(\bm{x}) < z^{*\iter{n}} - \alpha\tilde{\gamma}^{\iter{n}} \right\}.
\]

\noindent Whenever any $\bm{x}^{(n+1)} \in \itrval{\mathcal{E}}{n}{}$ passes the test \eqref{eqn:exploit-cond} and is thus sampled at iteration $n+1$, there are three possible situations:

\begin{enumerate}
    \item[1)] $\exists s \in [1,~S] : \idxval{c}{n+1}{s}<0$ (the new sample turns out to be unfeasible). Due to the violation of at least one constraint $g_s$, there is a region $\mathcal{B}(\bm{x}^{(n+1)},r_s)$, which is a ball around $\bm{x}^{(n+1)}$ with radius 

    \[
      r_s \triangleq \frac{\left|c^{(n+1)}_s\right|}{\tilde{\rho}^{\iter{n+1}}},
    \]
    
    such that $\forall \bm{x} \in \mathcal{B}(\bm{x}^{(n+1)},r_s), \tilde{g}_s(\bm{x}) < 0$, implying estimated violation w.r.t. $g_s$ and the exclusion from our eligible region for exploitation, i.e.
    
    \[
      \mathcal{E}^{\iter{n+1}} = \mathcal{E}^{\iter{n}} \setminus \mathcal{B}(\bm{x}^{(n+1)},r_s).
    \]
    
    \noindent This means that a region at least the size of the ball is removed from $\mathcal{E}$, which is valid even when $\tilde{\rho}^{\iter{n}}_s$ updates (with an increase or decrease), because $\underline{\rho}_s \leq \tilde{\rho}^{\iter{n}}_s \leq \rho_s$.

    \item[2)] $z^{(n+1)} \geq z^{*\iter{n}}$ (the new sample does not improve over the best one). In this case, we consider $\mathcal{B}(\bm{x}^{(n+1)},r_\theta)$ with 

    \[
      r_\theta \triangleq \frac{z^{(n+1)} - (z^{*\iter{n}} - \alpha \tilde{\gamma}^{\iter{n}})}{\tilde{\gamma}^{\iter{n+1}}}.
    \]

    \noindent Then $\forall \bm{x} \in \mathcal{B}(\bm{x}^{(n+1)}, r_\theta)$ it holds that  $ \underline{f}^{\iter{n+1}}(\bm{x}) \geq z^{*\iter{n}} - \alpha \tilde{\gamma}^{\iter{n}}.$ Therefore all the points inside the hyper-ball are not eligible for exploitation at iteration $n+1$. Hence, we have $\mathcal{E}^{\iter{n+1}} = \mathcal{E}^{\iter{n}} \setminus \mathcal{B}(\bm{x}^{(n+1)},r_\theta)$, thus again reducing the volume of the set of candidates that are eligible for exploitation.
    
    \item[3)] $z^{(n+1)} < z^{*\iter{n}}$ (the new sample becomes the best one). In this case, the set of candidate points for exploration shrinks, i.e. $\mathcal{E}^{\iter{n+1}} \subset \mathcal{E}^{\iter{n}}$ due to the new threshold. Moreover, the hyper-ball $\mathcal{B}(\bm{x}^{(n+1)},r_\theta)$, with $r_\theta = \alpha$ around the new sample is also removed from $\mathcal{E}$. 
    \end{enumerate}
    
    Since in all the cases the volume of the set $\mathcal{E}^{\iter{n}}$ decreases by a finite quantity, we have that $\mathcal{E}^{\iter{n}} = \varnothing$ after finite iterations, and Mode~$\theta$ will fail, proving the lemma.
\end{proof}


The following lemma, which is essential in building our convergence theorem, proves that the exploration will generate an increasingly dense distribution of points throughout the search space.

\begin{lemma}
\label{lemma:dense-pts}
    Assume that Mode~$\psi$ (exploration) is called infinitely often as $n \rightarrow \infty$. Then, for any $\hat{\bm{x}}\in\mathcal{X}$ and any $\sigma > 0$, $\exists n_\sigma < \infty$ such that
    
    \[ 
        \min_{\bm{x}^{(i)} \in \bm{X}^{\iter{n_\sigma}}} \| \bm{x}^{(i)} - \hat{\bm{x}} \| < \sigma. 
    \]
    
\end{lemma}

\begin{proof}
Consider any point $\hat{\bm{x}} \in \mathcal{X} \setminus \itrpt{X}{n}{}$ and its nearest sample

\begin{equation}
    \overline{\bm{x}} = \argmin_{\bm{x} \in \itrpt{X}{n}{}} \| \bm{x} - \hat{\bm{x}} \|,
\end{equation}

\noindent and define

\begin{equation}
    \mathcal{Q} \triangleq \left\{ \bm{x} \in \mathcal{X} ~\Big|~ (\bm{x} - \overline{\bm{x}})^\top(\hat{\bm{x}} - \overline{\bm{x}}) > 0 \right\}.
\end{equation}



Moreover, consider the point

\begin{equation}
    \label{eqn:tilde-x}
    \tilde{\bm{x}} \triangleq 
    \begin{cases}
    \argmin_{\bm{x} \in \mathcal{Q} \cap \itrpt{X}{n}{}} \| \bm{x} - \hat{\bm{x}} \| & \text{if }\mathcal{Q} \cap \itrpt{X}{n}{} \neq \varnothing, \\
    \argmax_{\bm{x} \in \mathcal{Q}} \|\bm{x} - \overline{\bm{x}}\| & \text{otherwise.}
    \end{cases}
\end{equation}

\noindent By construction, we have an empty half-ball 

\begin{equation}
\mathcal{H} = \mathcal{Q} \cap \mathcal{B}(\overline{\bm{x}}, \| \overline{\bm{x}} - \tilde{\bm{x}} \|)
\end{equation}

\noindent and we show that there will be a candidate point sampled within $\mathcal{H}$ after finite iterations. Due to the mechanism described in the algorithm, existence of candidate points within $\mathcal{H}$ is guaranteed in both cases for $\tilde{\bm{x}}$ in \eqref{eqn:tilde-x}. In the first case ($\mathcal{Q} \cap \itrpt{X}{n}{} \neq \varnothing$), this is due to those along the segment between $\overline{\bm{x}}$ and $\tilde{\bm{x}}$. For the other case, candidate points are still guaranteed to exist because of the ones generated in the coordinate directions. Now, we assume for the sake of contradiction that points within $\mathcal{H}$ are never sampled. 

We consider a candidate point $\bm{e} \in \mathcal{H}$. Furthermore, consider the ranking of candidate points in $\itrpt{E}{n}{}$, in decreasing order of the merit function  $\xi_\psi(\bm{x})$:

\[
    \itrpt{M}{n}{} \triangleq \left\{ \itrpt{m}{n}{1}, \itrpt{m}{n}{2}, \itrpt{m}{n}{3},\ldots \right\}.
\]

\noindent We recall the definition of $\xi_\psi(\bm{x})$ in \eqref{eqn:explore-merit} and the finite-boundedness of $h(\bm{x})$, 

\[
    \underline{h} \leq \itrval{h}{n}{}(\bm{x}) \leq \overline{h}.
\]


Assume that $\bm{e}$ is generated at iteration $n_g$. We first calculate the number of iterations $n_e$ such that from iteration $n_g+n_e$, no candidate points will be generated with merit greater than $\xi_\psi(\bm{e})$. Denoting a candidate point generated at $n_g+n_e$ as $\bm{e}'$, and noting that  $\itrval{\xi}{n_g+n_e}{\psi}(\bm{e}') \leq \overline{h}$ (age of $\bm{e}'$ is zero at time of generation) and $\itrval{\xi}{n_g+n_e}{\psi}(\bm{e}) \geq \underline{h}+\psi k(n_e)$, it applies that


\begin{equation}
\label{eqn:n-e}
    n_e \leq k^{-1}\left( \frac{\overline{h} - \underline{h}}{\psi}\right),
\end{equation}

\noindent where $k^{-1}$ is the inverse function of $k$. Now we count the number $M_>$ of candidate points in $\itrpt{M}{n_g+n_e}{}$ with greater merit than $\itrval{\xi}{n_g+n_e}{\psi}(\bm{e})$. The upper bound for such a number is when all generated candidate points from iteration $n_g$ to $n_g+n_e$ have merits exceeding that of $\bm{e}$, hence, recalling the candidate points generation mechanism, we have that

\[
    M_> \leq \left| \itrpt{M}{n_g}{} \right| + (B-1)\sum_{n=n_g}^{n_g+n_e} n\left(2D+\frac{n-1}{2}\right) < \infty.
\]

\noindent where $|\cdot|$ is the set cardinality operator. This implies that $\bm{e}$ emerges in front of the ranking $\itrpt{M}{n_g+M_>}{}$, and will be taken at the next exploration after $n_g+M_>$ iterations, falling into contradiction. After sampling $\bm{e} \in \mathcal{H}$, $\mathcal{H}$ can be redefined using $\tilde{\bm{x}} \leftarrow \bm{e}$, and the same reasoning is repeated to show that the considered $\mathcal{H}$ shrinks, and

\begin{equation}
\label{eqn:half-ball-shrinks}
    \lim\limits_{n\rightarrow\infty}\| \tilde{\bm{x}}^{\iter{n}} - \bar{\bm{x}}^{\iter{n}} \| = 0.
\end{equation}

From \eqref{eqn:half-ball-shrinks}, and using the arguments in Lemma~5 of \cite{Sabug2020}, it applies that
\[ 
    \lim\limits_{n \rightarrow \infty}\| \bar{\bm{x}}^{\iter{n}} - \hat{\bm{x}} \| = 0,
\]
\noindent and, for any $\sigma>0$, we have
\[
    \exists n_\sigma: \min_{\bm{x}^{(i)} \in \bm{X}^{\iter{n_\sigma}}} \| \bm{x}^{(i)} - \hat{\bm{x}} \| < \sigma,
\]

\noindent which completes the proof.
\end{proof}


Given the above lemmas, we are now ready to present the main convergence result for SMGO-$\Delta$.

\begin{theorem}
\label{theorem:optim-gap}
Let Assumptions \ref{ass:lipschitz} hold. Then,
$\forall \varepsilon>0,\;\exists n_\varepsilon<\infty: z^{*\iter{n_\varepsilon}} \leq z^* + \varepsilon$.
\end{theorem}

\begin{proof}
Consider a point $\bm{x}^* \in \left\{ \bm{x}' ~\Big|~ \bm{x}' = \argmin_{\bm{x} \in \mathcal{X}} f(\bm{x})\right\}$ and take $\sigma = \dfrac{\varepsilon}{\gamma}$. For any $n$, denote
\[
\bar{\bm{x}}^{\iter{n}}=\arg\min\limits_{\bm{x}^{(i)} \in \bm{X}^{\iter{n}}} \| \bm{x}^* - \bm{x}^{(i)} \|
\]
In virtue of Lemma \ref{lemma:exploit-will-fail}, the exploration mode will be called infinitely often. Now, by applying Lemma \ref{lemma:dense-pts}, we have that $\exists n_\sigma < \infty :  \| \bm{x}^* -\bar{\bm{x}}^{\iter{n_\sigma}}\| ~<~ \sigma=\dfrac{\varepsilon}{\gamma}$. 
Then, in virtue of Assumption \ref{ass:lipschitz} we have:
\[
f(\bm{x}^{*\iter{n_\sigma}})-f(\bm{x}^*)=z^{*\iter{n}}-z^*\leq f(\bar{\bm{x}}^{\iter{n_\sigma}})-f(\bm{x}^*)\leq\gamma\|\bm{x}^* - \bar{\bm{x}}^{\iter{n_\sigma}} \|<\varepsilon
\]
Thus proving the result with $n_\varepsilon=n_\sigma$.
\end{proof}

\section{Extensions and Implementation Aspects}
\label{sec:implement-notes}

\subsection{On bounded noise and disturbances}

Our previous work \cite{Sabug2020CDC,Sabug2020} considered cases in which black-box functions evaluations are assumed exact and without noise. However, in practice, the values of the objective (resp. constraint) function can be affected by additive disturbance ${\epsilon}_f$ (resp. ${\epsilon}_s$) for any test point $\bm{x}\in\mathcal{X}$, which we assume to be bounded:

\begin{align}
\label{eqn:sampling-f-noise}
    \forall n, ~ & z^{(n)} = f(\idxpt{x}{n}{}) + \epsilon_f, |\epsilon_f| \leq \overline{\epsilon}_f \\
\label{eqn:sampling-g-noise}
    ~ & \idxval{c}{n}{s} = g_s(\idxpt{x}{n}{}) + \epsilon_s, |\epsilon_s| \leq \overline{\epsilon}_s, s = 1,\ldots,S.
\end{align}

\noindent In this case, we can derive estimates of the disturbance bounds, $\tilde{\epsilon}_f$ and $\tilde{\epsilon}_s$, from $\itrpt{X}{n}{}$, resorting to the method proposed in \cite{Fagiano2016}:

\begin{equation}
\small
    \label{eqn:epsilon-f-update-noise}
    \itrval{\tilde{\epsilon}}{n}{f} = \frac{1}{n} \sum_{i =1}^n \epsilon^{(i)}_f,
\end{equation}

\begin{equation}
\small
    \label{eqn:epsilon-s-update-noise}
    \itrval{\tilde{\epsilon}}{n}{s} = \frac{1}{n} \sum_{i =1}^n \epsilon^{(i)}_s,
\end{equation}

\noindent where

\[
    \epsilon^{(i)}_f = \max_{(\boldsymbol{x},z,\boldsymbol{c}) \in \idxpt{J}{i}{}} \Big| \idxval{z}{i}{} - z\Big|,
\]

\[
    \epsilon^{(i)}_s = \max_{(\boldsymbol{x},z,\boldsymbol{c}) \in \idxpt{J}{i}{}} \Big| \idxval{c}{i}{s} - c_s \Big|
\]

\noindent and 
\[
    \idxpt{J}{i}{} \triangleq \Big\{ (\boldsymbol{x},z,\boldsymbol{c})\in \itrpt{X}{n}{} ~\Big|~ \| \idxpt{x}{i}{}-\boldsymbol{x} \| \leq \nu \Big\}
\]

\noindent for a chosen (small) radius $\nu > 0$. 
As a rule of thumb, one can set $\nu = 0.1 ~d(\mathcal{X})$, where $d(\mathcal{X})$ is the diameter of the search space. If $\idxpt{J}{i}{} = \varnothing\,\forall i=1,\ldots,n$, one can increase $\nu$ and repeat the calculations. Furthermore, the algorithm now updates the Lipschitz constants' estimates $\tilde{\gamma}^{\iter{n}}$ and $\tilde{\rho}_s^{\iter{n}},\,s=1,\ldots,S$ (see Assumption \ref{ass:lipschitz}) as in \cite{Fagiano2016}:
\vspace{-0.125in}

\begin{equation}
\small
    \label{eqn:gamma-update-noise}
    \itrval{\tilde{\gamma}}{n}{} = \max_{(\idxpt{x}{i}{},\idxval{z}{i}{},\idxpt{c}{i}{}), (\idxpt{x}{j}{},\idxval{z}{j}{},\idxpt{c}{j}{})\in\itrpt{X}{n}{}} \left( 
    \begin{cases}
      \frac{|\idxval{z}{i}{}-\idxval{z}{j}{}|-2\itrval{\tilde{\epsilon}}{n}{f}}{\|\idxpt{x}{i}{}-\idxpt{x}{j}{}\|} & \text{if~} |\idxval{z}{i}{} - \idxval{z}{j}{}| \geq 2\itrval{\tilde{\epsilon}}{n}{f}\\
      \underline{\gamma} & \text{otherwise}
    \end{cases} \right),
\end{equation}

\begin{equation}
\small
    \label{eqn:rho-update-noise}
    \itrval{\tilde{\rho}}{n}{s} = \max_{(\idxpt{x}{i}{},\idxval{z}{i}{},\idxpt{c}{i}{}), (\idxpt{x}{j}{},\idxval{z}{j}{},\idxpt{c}{j}{})\in\itrpt{X}{n}{}} \left( 
    \begin{cases}
      \frac{|\idxval{c}{i}{s}-\idxval{c}{j}{s}|-2\itrval{\tilde{\epsilon}}{n}{s}}{\|\idxpt{x}{i}{}-\idxpt{x}{j}{}\|} & \text{if~} |\idxval{c}{i}{s} - \idxval{c}{j}{s}| \geq 2\itrval{\tilde{\epsilon}}{n}{s}\\
      \underline{\rho}_s & \text{otherwise}
    \end{cases} \right).
\end{equation}

\begin{remark}
    The theoretical results laid out in Lemmas~\ref{lemma:exploit-will-fail} and \ref{lemma:dense-pts} still hold valid even when considering the case with bounded noise, that is, the exploitation routine will end in finite iterations, and the algorithm will densely cover the search set. 
    In fact, since the update of $\itrval{\tilde{\epsilon}}{n}{f}$ applies to all sample points $\idxpt{x}{i}{}$ when calculating $\itrval{\underline{f}}{n}{}(\bm{x})$ (see \eqref{eqn:f-lower-bounds}), any increase or decrease of such an estimate will only shift the lower bounds surface by a constant term. Hence, all results discussed above still hold valid. Furthermore, Lemma~\ref{lemma:dense-pts} is anchored to the candidate points generation and the age-based term in the exploration merit ranking, hence the dense points argument still holds.
\end{remark}

\subsection{On the choice of search directions}
The proposed candidate points generation scheme is simple and reproducible, but it may limit the possible directions to sample throughout the search space. Hence, we introduce, in addition to the candidate points generation routine described in Section \ref{subsec:cdpt-generation}, a pseudo-random points distribution throughout $\mathcal{X}$. As in \cite{Eriksson2020}, we use a $D$-dimensional Sobol sequence to generate such points.

In this case, the set of exploration candidate points $\itrpt{E}{0}{}$ is initialized as a finite set of $L$ points from the Sobol sequence (in this paper, we set $L=500$). From the first sample onward, $\itrpt{E}{n}{}$ is augmented with the new points as in Section \ref{subsec:cdpt-generation}. Note that these initial candidate points share the properties of all other succeeding ones:

\begin{itemize}
    \item fixed candidate point locations from the time of generation, and
    \item increasing age w.r.t. $n$,
\end{itemize}

\noindent which maintains the validity of all the theoretical properties discussed in Section~\ref{sec:algo-analysis}.

A similarly-generated set of $L$ pseudo-random candidate points is also used in exploitation to improve the coverage of the trust region. As $\mathcal{T}^{\iter{n}}$ expands or contracts, the pseudo-random points proportionally scale in the same manner. 

They not only increase the search directions in the exploitation step, but more importantly guarantee the existence of at least $L$ candidate points inside the thrust region $\mathcal{T}^{\iter{n}}$, even if it becomes small. The trust region filler points change locations w.r.t. iterations, depending on the size and location of $\mathcal{T}^{\iter{n}}$. Nevertheless, the arguments laid out in Lemma~\ref{lemma:exploit-will-fail} still hold true, because they do not require any assumption on the location of the exploitation candidate points. Furthermore, assuming the same seed in the pseudo-random candidate points generation, we still maintain the reproducibility property of SMGO-$\Delta$.

\subsection{On the computational complexity}
The computational complexity of the proposed algorithm SMGO-$\Delta$ is lower than that of the SMGO discussed in \cite{Sabug2020} regardless of consideration of additive uncertainty. This is mainly because of the candidate points' generation mechanism, which leads now to $\mathcal{O}(Dn + n^2)$, as compared with $\mathcal{O}(2^D + n^2)$ from \cite{Sabug2020CDC,Sabug2020} for a $D$-dimensional hyperrectangle search space.

In the noiseless case, the Lipschitz constant update can be done in an iterative manner as in \cite{Sabug2020}. Similarly, the routines to calculate the (upper- and lower-) bounds and the central estimates of $f(\bm{x})$ and $g_s(\bm{x})$ can take an iterative implementation.

When noise on the objective and/or constraints should be considered, the noise amplitudes estimation entail the recalculation of the radius $\nu$ with which the sample differences are calculated, hence this cannot be done iteratively in general. As the Lipschitz constant calculations depend on the estimated noise amplitudes, these also have to be repeated for each iteration. As both these processes require comparing each sample with all the other ones (and the comparison is linear w.r.t. $D$), these processes have complexity $\mathcal{O}(Dn^2)$, which is still polynomial.

The introduction of pseudo-random points as candidates for exploration does not change the complexity of the whole algorithm, since these are treated like all other candidate points in terms of the calculations required. On the other hand, the introduction of pseudo-random points within $\mathcal{T}^{\iter{n}}$ for exploitation requires the recalculation of \eqref{eqn:exploit} each time the best point or the size of the thrust region are modified. Nevertheless, the computations needed for a fixed number of points do not impact the complexity.




\section{Performance and Sensitivity Analysis}
\label{sec:perf-test}
We now analyze, in an illustrative low-dimensional problem, the behavior of SMGO-$\Delta$, focusing on its sensitivity w.r.t. the user-defined parameters involved.

\subsection{Illustrative example}
We consider the optimization of a 2D function with 2 constraints, see Tables~\ref{table:visual-fn-enum}-\ref{table:visual-fn-specs}, using $\alpha=0.005$ and $\Delta=0.25$. Fig.~\ref{fig:visual-contour} shows the search space with the objective function contours. The shaded regions are the respective feasible regions for $g_1$ and $g_2$, resulting in disjoint feasible regions composed of a circle and portions of circular strips. The global minimum is marked with an $\otimes$, while other local minima are shown with $\odot$.


The history of sampled values $\idxval{z}{n}{}$ and the best value $z^{*\iter{n}}$ is shown in Fig.~\ref{fig:visual-hist}, along with the sample feasibility and the operation mode taken. The $z^{*\iter{n}}$ update is solely done when a better \textit{feasible} sample is found. Hence, there were samples with improved costs but found in violation with one or more constraints, and did not change the previous $z^{*\iter{n}}$.

Fig.~\ref{fig:visual-pts-distrib} shows the resulting spread of samples in the search space. Mostly, the samples were acquired within the feasible area, and exhibited a concentration around the global optimal point. This is mainly due to a relatively low improvement threshold factor $\alpha$, leading to more exploitation points being taken (chosen from candidate points in the estimated feasible region). On the other hand, the exploration iterations were used to sample the candidate points within the feasible region, due to a low $\Delta$ parameter, leading to a conservative exploration behavior. However, notice that even with such conservatism, there were still a sizeable number of samples which fell outside the feasible region. This is due to the \textit{a priori} unknown noise bounds and Lipschitz constants for the constraints, both of which have to be estimated from data.

\begin{table*}[!t]
\small
    \begin{center}
    	\begin{tabular}{|c|>{\centering\arraybackslash}m{2in}|p{1.5in}|}
    		\hline
    		\textbf{Description} & \textbf{Function definition} & \textbf{Comments}  \parbox{0pt}{\rule{0pt}{2ex+\baselineskip}}\\ \hline
    		Objective $f(\bm{x})$ \parbox{0pt}{\rule{0pt}{5ex+\baselineskip}} & $\frac{1}{2} \sum_{i=1}^2 \left( x_i^4 - 16x_i^2 + 5x_i \right) + 80$ & Styblinski-Tang function with offset (multimodal, separable) \\ \hline
    		Constraint $g_1(\bm{x})$ \parbox{0pt}{\rule{0pt}{5ex+\baselineskip}} & $-4 + \Big\| ~\bm{x} - \begin{bmatrix} -2.90 & 2.90 \end{bmatrix}^\top\Big\|$ & Upwards-facing cone centered at $(-2.90, 2.90)$ \\ \hline
    		Constraint $g_2(\bm{x})$ \parbox{0pt}{\rule{0pt}{5ex+\baselineskip}} & $\mathrm{cos}\left( 2~\Big\| ~\bm{x} + \begin{bmatrix} 2.90 & 2.90 \end{bmatrix}^\top\Big\| \right)$ & Concentric ripples around $(-2.90, -2.90)$ \\ \hline
    	\end{tabular}
        \caption{Functions used for illustrative tests}
    	\label{table:visual-fn-enum}
    \end{center}
\end{table*}

\begin{table*}[!t]
\small
    \begin{center}
    	\begin{tabular}{|p{1.5in}|c|}
    		\hline
    		\textbf{Description} & \textbf{Value}\\ \hline
    		Search space $\mathcal{X}$ & $-5 \leq x_i \leq +5, i=1,2$ \\ \hline
    		Starting point $\idxpt{x}{1}{}$ & (0.4775, 0.0667) \\ \hline
    		Maximum iterations $N$ & 500 \\ \hline
    		$f(\bm{x})$ noise amplitude $\overline{\epsilon}_f$ & 0.25 (white noise) \\ \hline
    		$g_1(\bm{x})$ noise amplitude $\overline{\epsilon}_1$ & 0.1 (white noise) \\ \hline
    		$g_2(\bm{x})$ noise amplitude $\overline{\epsilon}_2$ & 0.05 (white noise) \\ \hline
    	\end{tabular}
        \caption{Test parameters for the illustrative example}
    	\label{table:visual-fn-specs}
    \end{center}
\end{table*}

\begin{figure}[!t]
	\centering
 	\includegraphics[width=0.625\columnwidth]{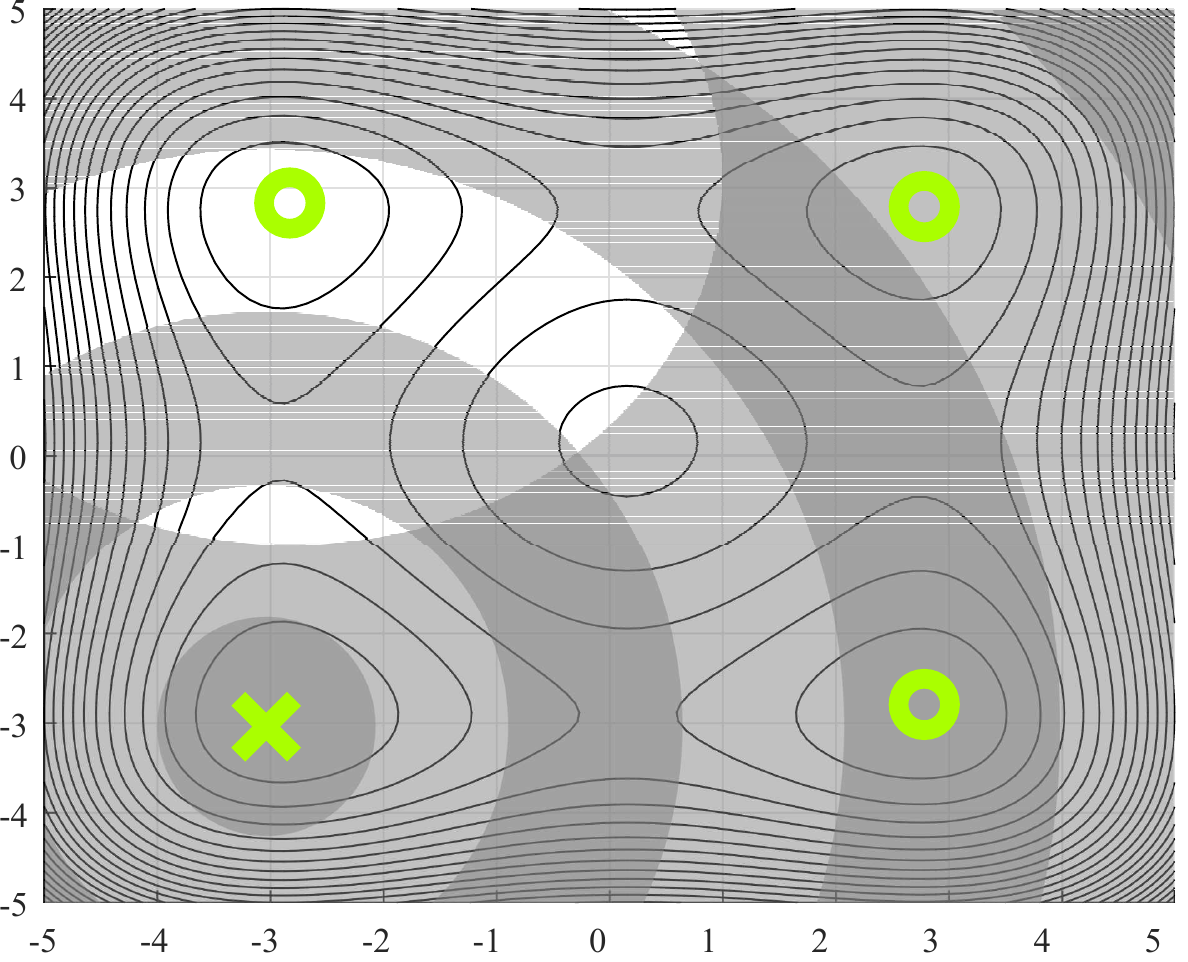}
	\caption{Illustrative example: contour plot of the objective function, \\ with constraint satisfaction regions}
	\label{fig:visual-contour}
\end{figure}

\begin{figure}[!t]
	\centering
 	\includegraphics[width=0.75\columnwidth]{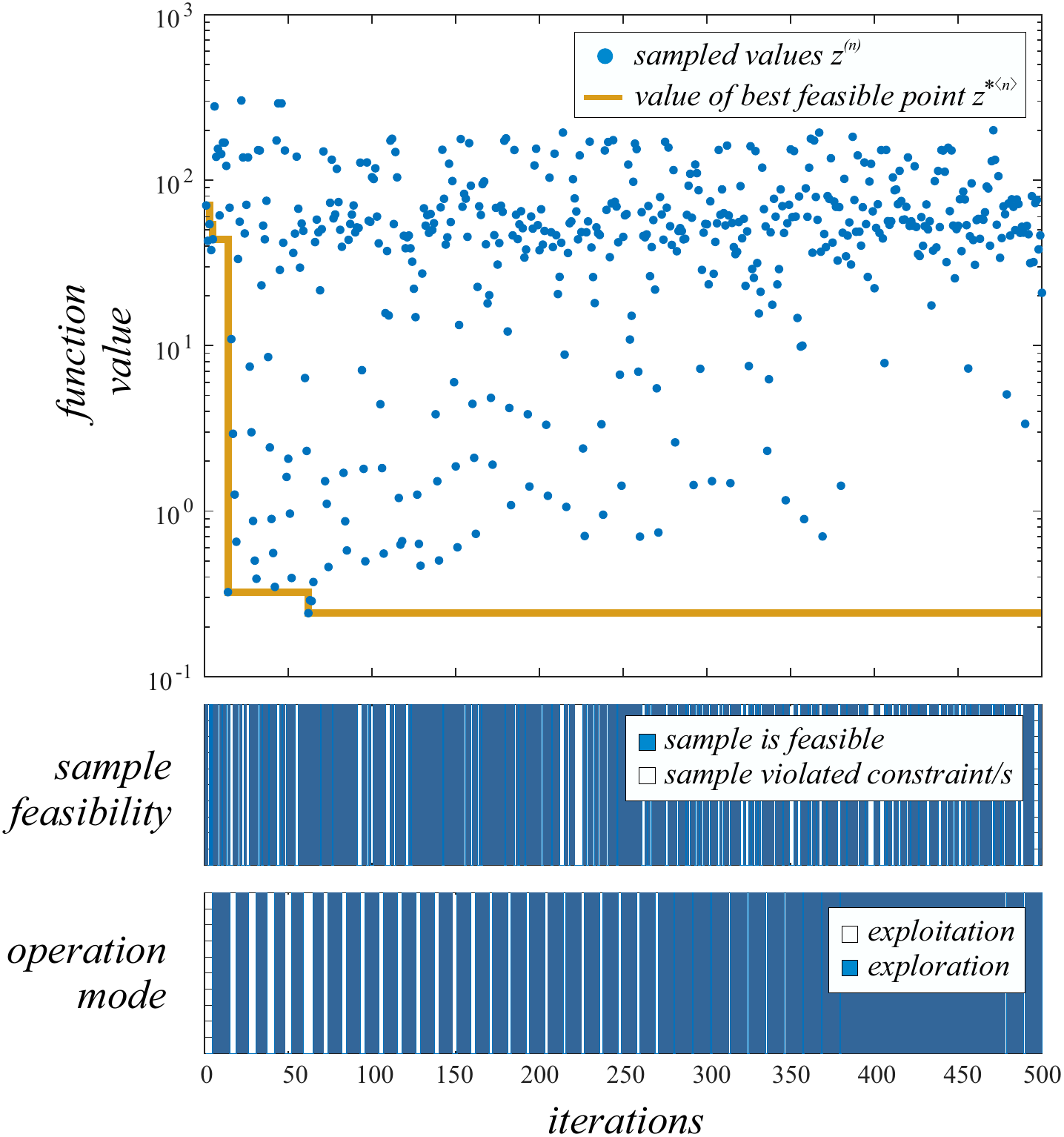}
	\caption{Illustrative example: sampled values history, best value history, test point feasibility, and SMGO-$\Delta$ mode taken (exploitation or exploration)}
	\label{fig:visual-hist}
\end{figure}

\begin{figure}[!t]
	\centering
 	\includegraphics[width=0.5\columnwidth]{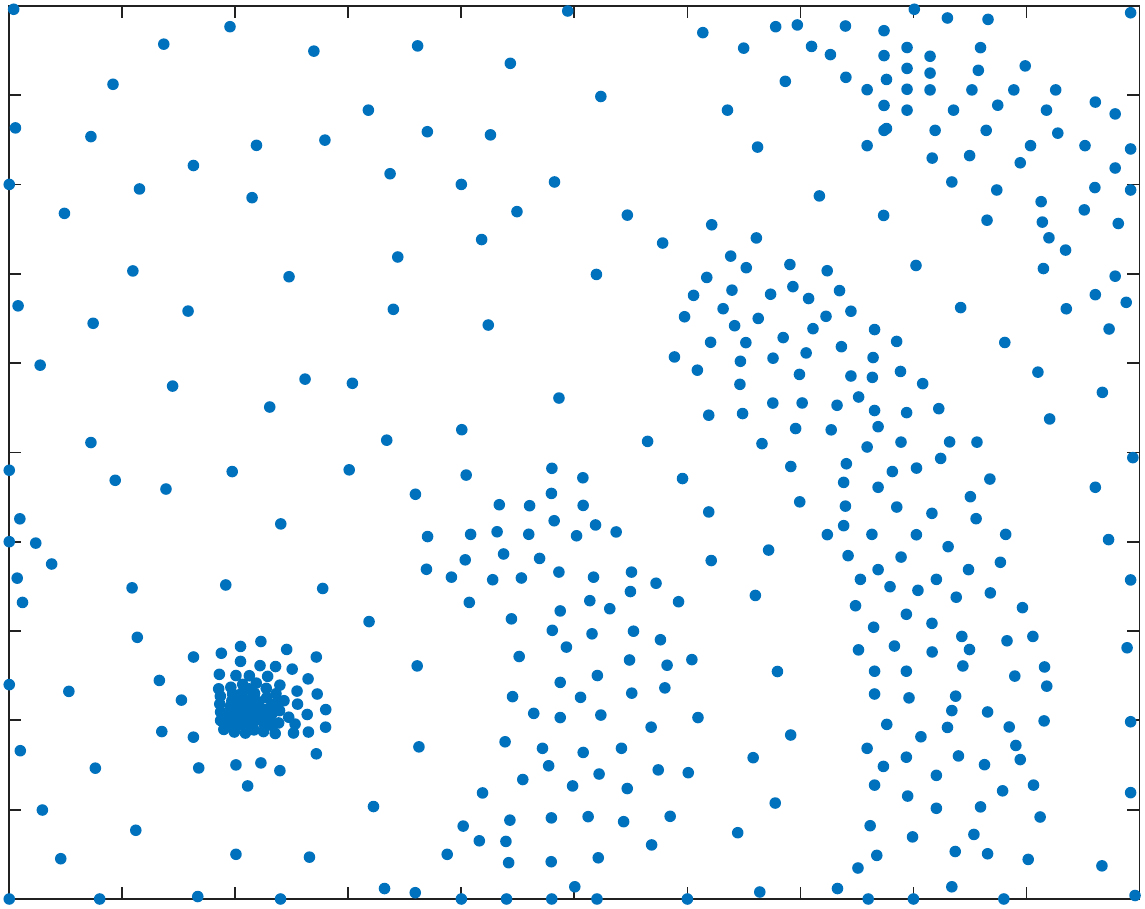}
	\caption{Illustrative example: resulting sample points distribution}
	\label{fig:visual-pts-distrib}
\end{figure}




\subsection{Influence of Risk Factor $\Delta$}

Now we show the effects of the $\Delta$ parameter on the exploration behavior of SMGO-$\Delta$. Since we want to focus on the exploration behavior, we set $\alpha=100$ to completely suppress the exploitation part of the algorithm. Furthermore, the $\Delta$ value is varied from 0.0 to 1.0, in increments of 0.25, while keeping all other hyperparameters constant. The resulting sample points distributions are shown in Fig.~\ref{fig:visual-delta}. In general, more points are sampled in the unfeasible region when using higher $\Delta$ value. However, for too small $\Delta$ (below 0.25 in this example), the algorithm is not able to discover part of the feasible region, such as the one containing $\bm{x}^*$. At the limit of $\Delta=1.0$, we can observe a quasi-uniform points distribution, however the samples density is higher in the feasible regions compared to those where constraints are violated.

\begin{figure}[!t]
	\centering
 	\includegraphics[width=\columnwidth]{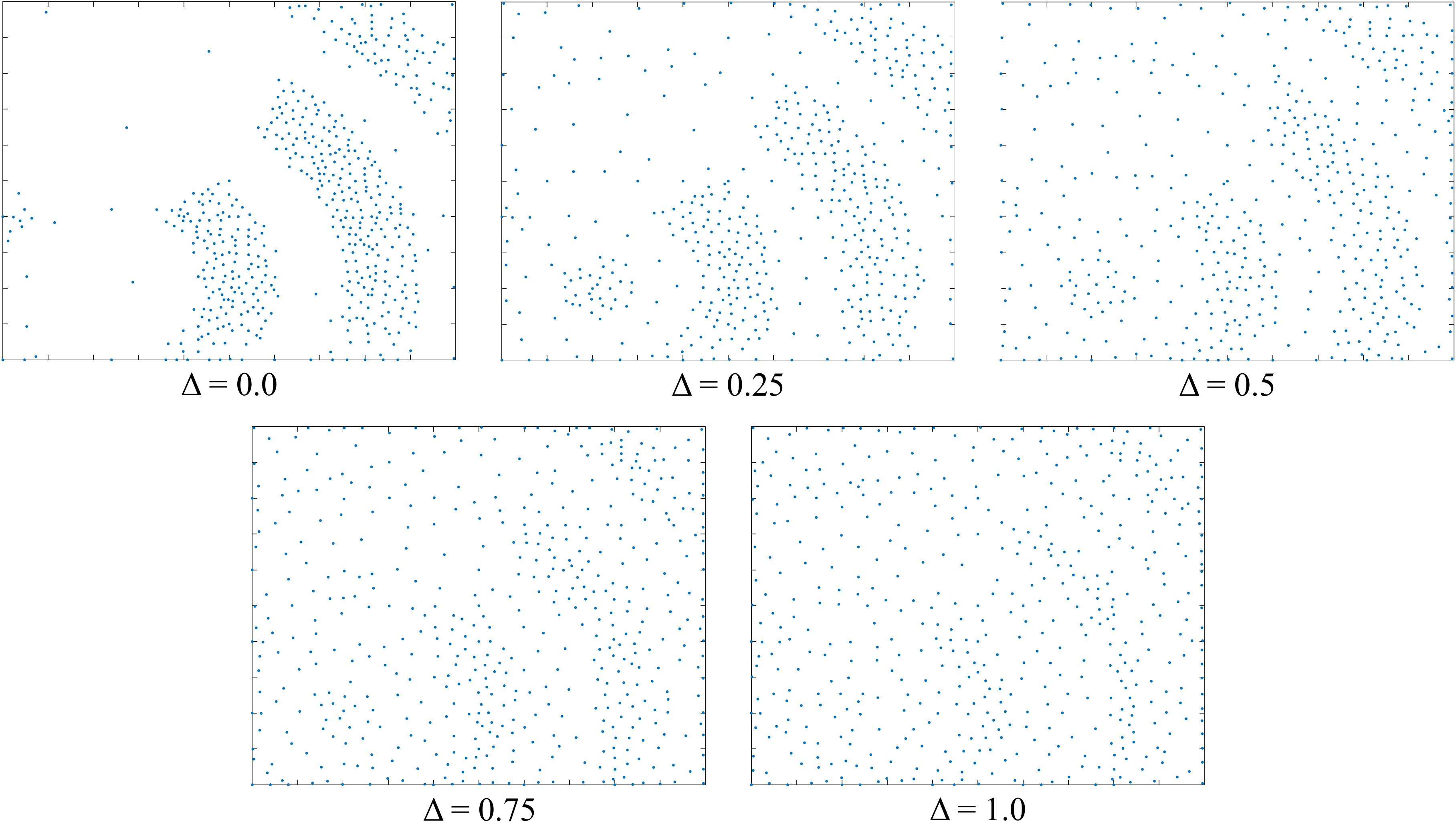}
	\caption{Illustrative example: results of varying the risk factor $\Delta$}
	\label{fig:visual-delta}
\end{figure}


\subsection{Sensitivity w.r.t. $\alpha$ and $\phi$}
Several trends of the algorithm behavior are investigated when the main parameters are varied as follows:

\begin{itemize}
    \item $\alpha \in \big\{ 0.001,~ 0.0025,~ 0.005,~ 0.01,~ 0.025,~ 0.05 \big\}$
    \item $\Delta \in \big\{ 0.0,~ 0.1,~ 0.2,~ \ldots,~ 0.9,~ 1.0 \big\}$
\end{itemize}

Each combination $(\alpha_i,~ \Delta_j)$ is tested by running 100 independent tests of the problem described by Tables~\ref{table:visual-fn-enum}-\ref{table:visual-fn-specs}. The starting points across these 100 runs (for the same $(\alpha_i,~ \Delta_j)$ combination) are randomly generated. However, the set of starting points are the same across different $(\alpha_i,~ \Delta_j)$ combinations to ensure uniformity in the tests.

Fig.~\ref{fig:results-vis-mean} shows the contour graph of the average optimality gap, measured as $z^{*\iter{N}} - z^*$, for the different combinations of $\alpha$ and $\Delta$. Fig.~\ref{fig:results-vis-pct-exploit} presents the contour of the average percentages of exploitation points taken throughout the runs. On the other hand, Fig.~\ref{fig:results-vis-pct-unfeas-total} shows the (average) total percent of unfeasible points sampled by the algorithm.

As expected, the percentage of exploitation points decreases w.r.t. increasing $\alpha$, due to a higher tendency to switch to exploration; this trend is verified by the mostly horizontal contours in Fig.~\ref{fig:results-vis-pct-exploit}. Less-expected observations can be taken from Fig.~\ref{fig:results-vis-pct-unfeas-total}, regarding the relation of $\alpha$ and $\Delta$ to the ``cautiousness'' of the algorithm. In a general case, we see in Fig.~\ref{fig:results-vis-pct-unfeas-total} that the percentage of unfeasible samples increases with the increment of both $\alpha$ and $\Delta$. Also we can see that there are two approaches to increase cautiousness: to decrease $\alpha$ (exploit more), or to decrease $\Delta$, as evident in the vertical contours on the left side, and the horizontal contours on the bottom part of Fig.~\ref{fig:results-vis-pct-unfeas-total}.



True enough, adjusting $\Delta$ to larger values has its own incentives and trade-offs. As shown in Fig.~\ref{fig:results-vis-mean}, the optimality gap w.r.t. $\bm{x}^*$ improves with higher $\Delta$, especially with relatively large values of $\alpha$. However, as discussed, this is at the cost of increasing the unfeasible samples count. This trade-off introduces a perspective of balancing \textit{risk} and \textit{reward}, which can be tuned by the user. For this particular case, the use of a riskier exploration (higher $\Delta$ parameter) had its incentives. This is because as shown in Fig.~\ref{fig:visual-pts-distrib}, the feasible regions are disjoint, and the two local minima (one of which is the global one) are located in separate and faraway regions, rewarding a more aggressive exploration. In some cases, we can decide to have a more conservative exploration (choose a smaller $\Delta$) when given a relatively small iteration budget w.r.t. search space, e.g. when the problem is relatively high dimensional, or when large constraint violations shall be avoided. However, as pointed out in Remark~\ref{remark:safe-opt}, a completely safe optimization is not guaranteed even when $\Delta \rightarrow 0.0$, due to the lack of knowledge about the true Lipschitz constants $\rho_s$, which holds for a majority of contexts.

In choosing a good combination of $\alpha$ and $\Delta$ for our next experiments, we recognize in Fig.~\ref{fig:results-vis-mean} that the mean of the optimality gap improves from the bottom left corner going to the top right, but there has been no significant improvement in the region where $\alpha > 0.005$ and $\Delta > 0.20$. However in the corresponding area in Fig.~\ref{fig:results-vis-pct-unfeas-total}, increasing $\alpha$ and $\Delta$ is only paid by more unfeasible samples, which is not a reasonable trade-off. Hence, based on the results for this particular problem, we choose the combination of $\alpha = 0.005$ and $\Delta = 0.20$ for the studies presented next. We now summarize the resulting SMGO-$\Delta$ hyperparameters in Table~\ref{table:smgo-d-hyperparams}, used in the succeeding tests, and recommended as default values. 

\begin{table*}[!t]
\small
    \begin{center}
    	\begin{tabular}{|c|p{2.5in}|c|}
    		\hline
    		\textbf{Parameter} & \textbf{Description} & \textbf{Default value}\\ \hline
    		$\alpha$ & Expected improvement threshold factor & 0.005 \\ \hline
    		$\Delta$ & Risk parameter & 0.20 \\ \hline
    		$\beta$ & Weighting factor between $\tilde{f}^{\iter{n}}(\bm{x})$ and $\lambda^{\iter{n}}(\bm{x})$ for exploitation cost (see Section~\ref{subsec:exploit}) & 0.1 \\ \hline
    		$\phi$ & Candidate point age-based merit growth rate & $1\times 10^{-6}$ \\ \hline
    		$B$ & Multiplier for generated candidate points per iteration (see Section~\ref{subsec:cdpt-generation}) & 5 \\ \hline
    		$L$ & Number of generated pseudo-random candidate points for exploitation and resp. exploration & 500 \\ \hline
    		$\overline{\upsilon}$ & Trust region maximum side measure & 0.1 \\ \hline
    		$\kappa$ & Shrink factor for the trust region & 0.5 \\ \hline
    		$\underline{\upsilon}$ & Trust region minimum side measure & $\kappa^{10}\overline{\upsilon}$ \\ \hline
    	\end{tabular}
        \caption{Default SMGO-$\Delta$ hyperparameters summary}
    	\label{table:smgo-d-hyperparams}
    \end{center}
\end{table*}

\begin{figure}[!t]
	\centering
 	\includegraphics[width=0.65\columnwidth]{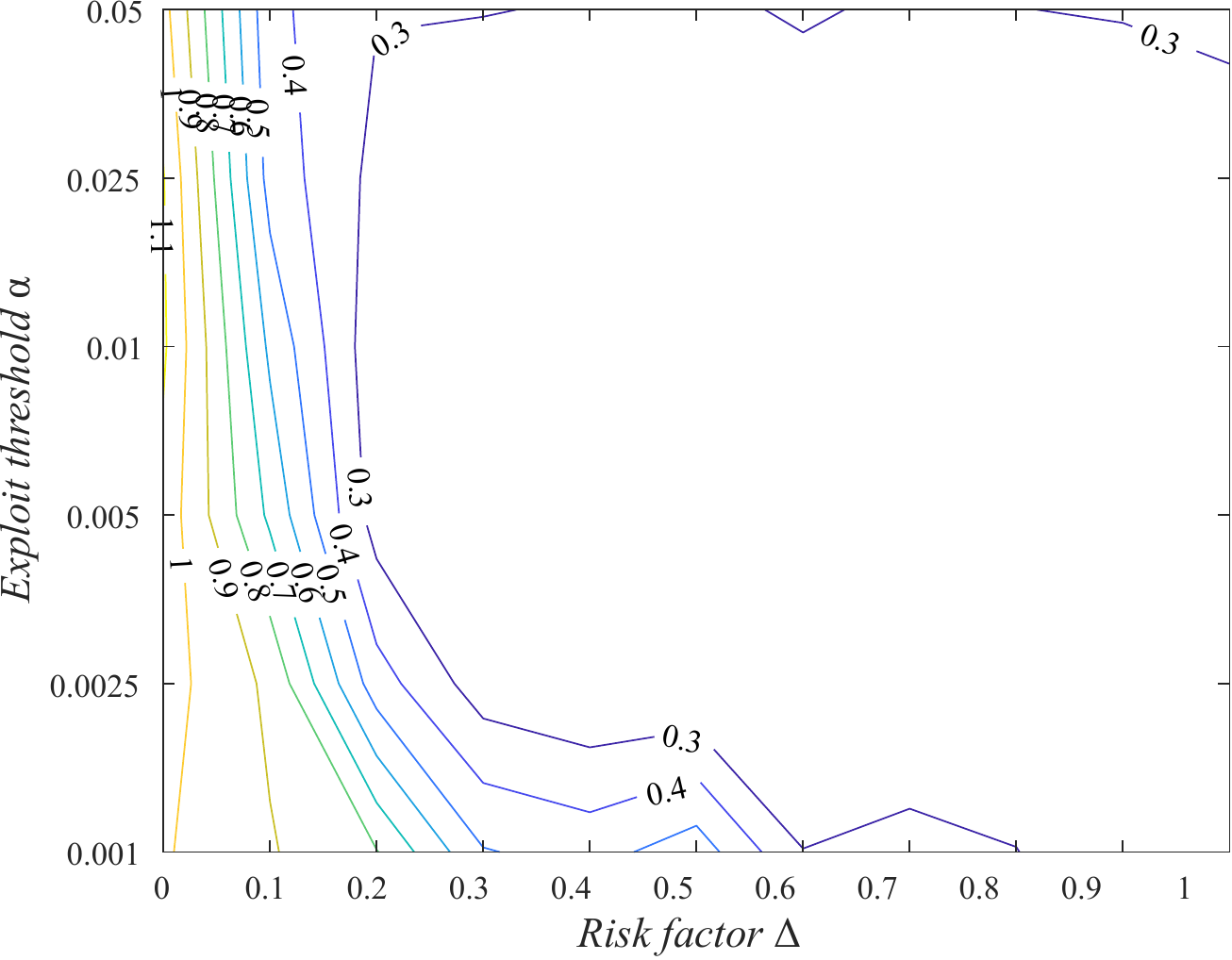}
	\caption{Illustrative example: contour plots of optimality gap, \\variations w.r.t. $\alpha$ and $\Delta$}
	\label{fig:results-vis-mean}
\end{figure}

\begin{figure}[!t]
	\centering
 	\includegraphics[width=0.65\columnwidth]{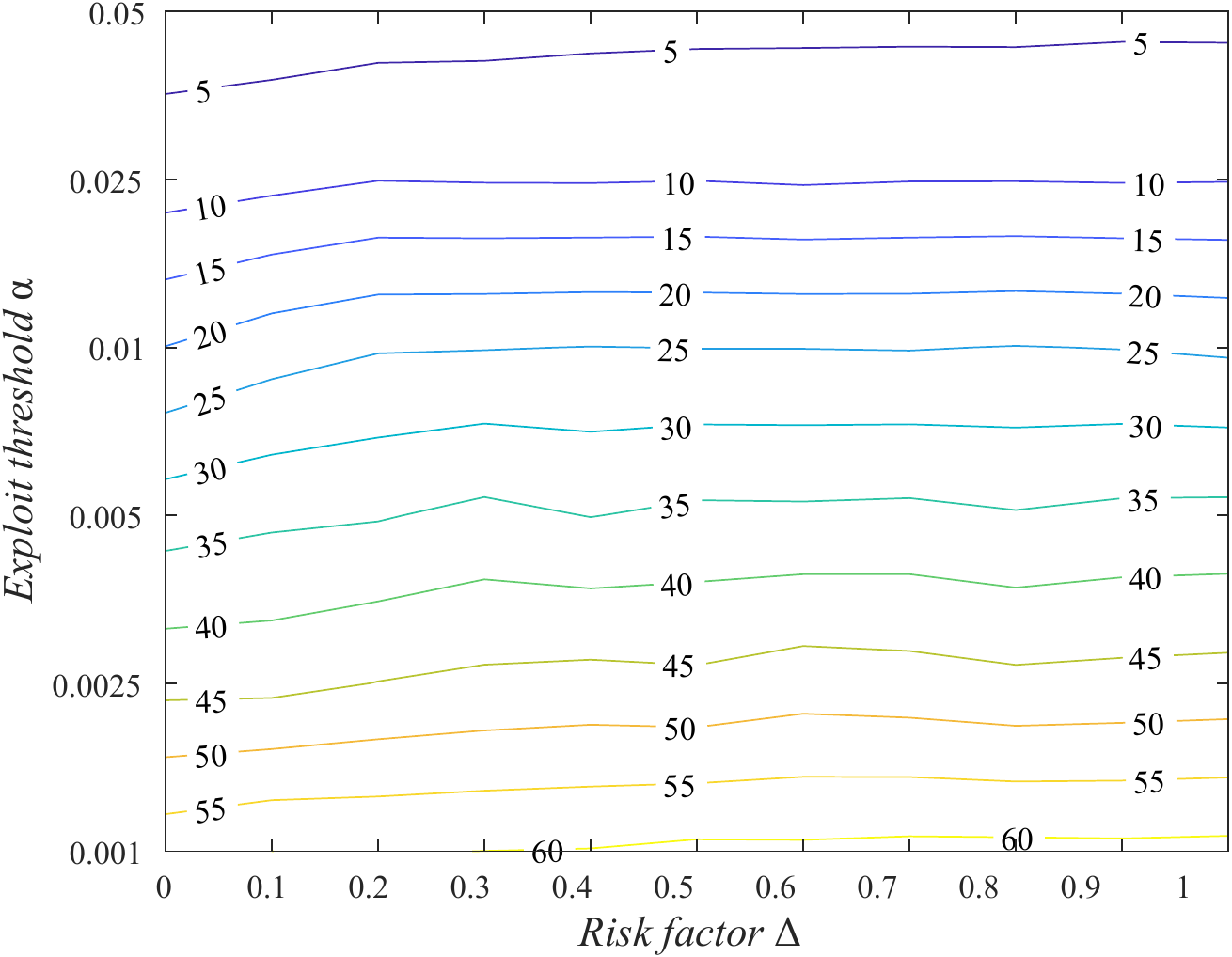}
	\caption{Illustrative example: average percentage of generated exploitation points, \\variations w.r.t. $\alpha$ and $\Delta$}
	\label{fig:results-vis-pct-exploit}
\end{figure}

\begin{figure}[!t]
	\centering
 	\includegraphics[width=0.65\columnwidth]{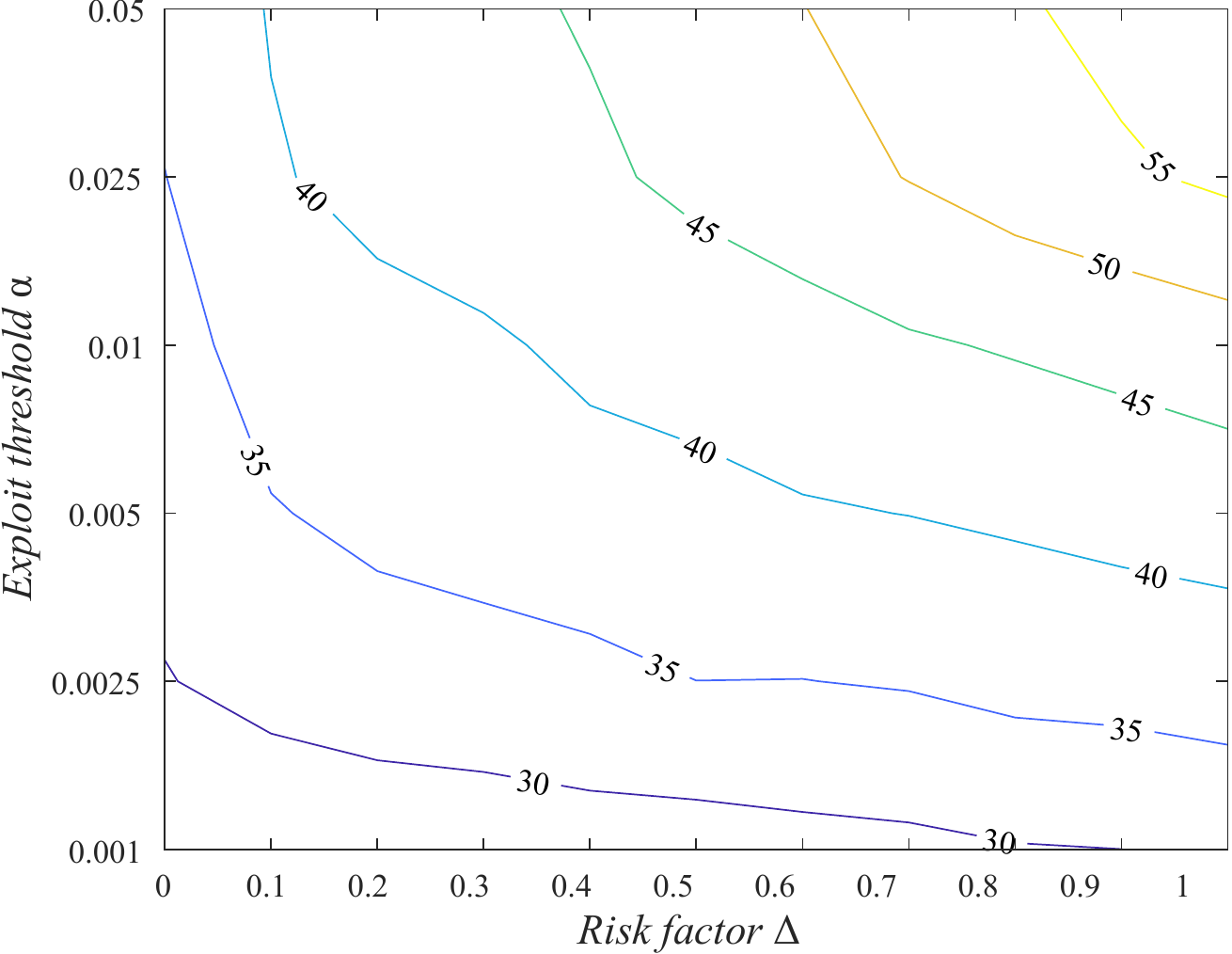}
	\caption{Illustrative example: average percentage of unfeasible samples, \\variations w.r.t. $\alpha$ and $\Delta$}
	\label{fig:results-vis-pct-unfeas-total}
\end{figure}

\section{Benchmarks}
\label{sec:syn-tests}

We now compare SMGO-$\Delta$ with representative existing methods for black-box optimization with black-box constraints. The first method is the constrained Bayesian optimizer (CBO) as available in MATLAB Optimization Toolbox, the most common approach for black-box optimization in the recent literature. Another competitor method we used is the Nonlinear Optimization by Mesh Adaptive Direct Search (NOMAD) \cite{ledigabel2011,audet2021nomad}, another widely-used algorithm, of which the latest version NOMAD 4 \cite{audet2021nomad} is made available by its authors. We use the default settings for the competitor algorithms, while the settings for SMGO-$\Delta$ are as chosen in the last section.

We subject all test algorithms to a benchmark with 10 synthetic problems, with different properties of the objective and constraint functions. For each problem and test algorithm, we performed a statistical test comprising of 50 independent runs, each run having 500 iterations. To preserve fairness of comparison, all test algorithms for the same problem use the same set of starting points for the respective independent runs.

\subsection{Benchmark problems}
A summarized overview of the synthetic problems we consider are given in Table~\ref{table:syn-tests}. The first 7 problems are from the CEC2006 benchmark set \cite{Liang2006}, and are commonly-used in benchmarking black-box optimization methods. Furthermore, T1 is a test problem used in \cite{Ariafar2019,Hernandez-Lobato2015,Hernandez-Lobato2016}, while T2 and T3 are from \cite{Ariafar2019,Gardner2014}. More details on the function definitions are provided in the Appendix. Note that some inequality constraints are enumerated in the form $g(\bm{x}) \leq 0$ to keep consistent with their respective source literature, but these have been appropriately adjusted to follow $g(\bm{x}) \geq 0$ when used with SMGO-$\Delta$.

Two problems denoted as G05MOD and G23MOD are modified versions of G05 and G23 found in \cite{Liang2006}. G05MOD takes from G05, but its equality constraints are converted to inequality constraints. G23MOD, on the other hand, is based on G23 \cite{Liang2006} but the equality constraints are removed. 

We have selected the benchmark problems to evaluate the performance of the compared optimizers in different dimensionality $D$, feasible region coverage $\rho$, as well as characteristics of the objective and constraint/s. Some functions, like G04 and G24 have fairly large feasible region coverage, and while G08 and G09 have coverage of less than 1\%. Other functions did not have estimated $\rho$ as available in their respective source literature.

\begin{table}
\small
\begin{center}
	\begin{tabular}{|c|c|c|c|c|c|c|}
		\hline
		\textbf{Problem} & $D$ & $S$ & $f$ \textbf{type} & $g$ \textbf{type} & $z^*$ & $\rho$ \\ \hline
		G04 & 5 & 6 & Q & NL & -3.0665e+04 & 52.123\% \\ \hline
		G05MOD & 4 & 5 & C & L+NL & 5.1265e+03 & N/A \\ \hline
		G08 & 2 & 2 & NL & NL & -0.0958 & 0.8560\% \\ \hline
		G09 & 7 & 4 & NL & NL & 680.6301 & 0.5121\% \\ \hline
		G12 & 3 & 1 & Q & NL & -1.0 & 4.7713\% \\ \hline
		G23MOD & 9 & 2 & L & NL & N/A & N/A \\ \hline
		G24 & 2 & 2 & L & NL & -5.5080 & 79.6556\% \\ \hline
		T1 & 2 & 2 & L & Q+NL & N/A & N/A \\ \hline
		T2 & 2 & 1 & NL & NL & N/A & N/A \\ \hline
		T3 & 2 & 1 & NL & NL & N/A & N/A \\ \hline
	\end{tabular}
	\caption{Summary of benchmark problems. $\rho$ stands for the estimated feasible region coverage w.r.t. search space. L: linear, Q: quadratic, C: cubic, NL: nonlinear. \\N/A: quantities are not specified in source literature}
	\label{table:syn-tests}
\end{center}
\end{table}

\subsection{Graphical comparison}

Figs.~\ref{fig:t1}-\ref{fig:t3} show the best points achieved by CBO, NOMAD~4, and SMGO-$\Delta$ from the different runs. In these plots, the greyed areas are the unfeasible regions, and the boundaries of the feasible region are drawn in solid black line. We see that all the results of CBO and SMGO-$\Delta$ lie around the same region in the search space. Furthermore, the spread of the SMGO-$\Delta$ best points are more spread than those of CBO, as seen in the zoom inset of all plots. In the same insets, NOMAD~4 is seen as having the most concentrated distribution, but there are also best points in all plots which lie on a faraway area in the search space. For example, for T2 (Fig.~\ref{fig:t2}), the NOMAD~4 points on the upper left oblong means that at some runs, the feasible region on the lower right (containing the optimal point) is not found. On another case, some NOMAD~4 results for T3 lie on the local minima (center left side, and upper right side of Fig.~\ref{fig:t3}).

\begin{figure}[!t]
	\centering
 	\includegraphics[width=0.85\columnwidth]{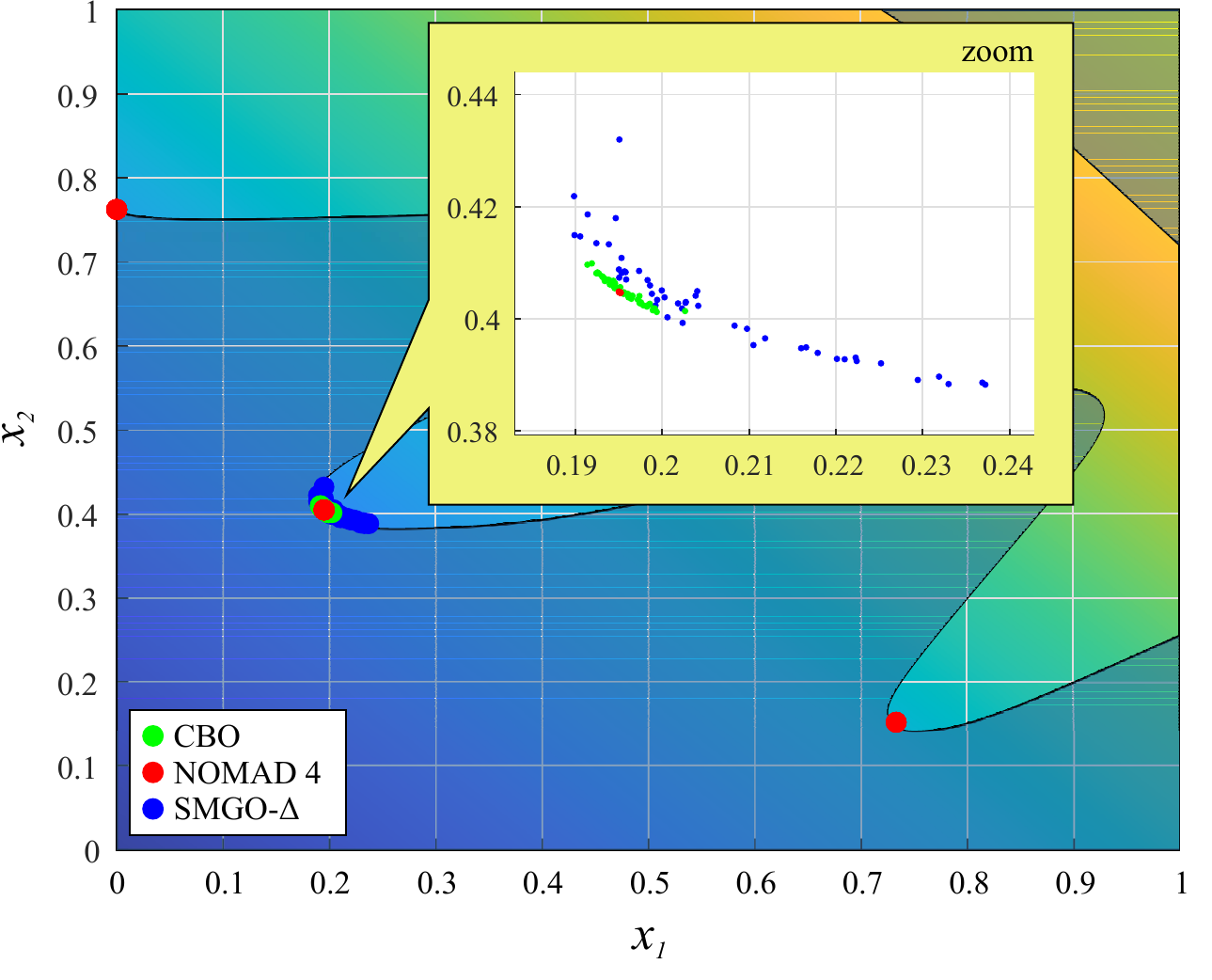}
	\caption{Benchmark tests: distribution of best results for \\ different optimization runs on T1.}
	\label{fig:t1}
\end{figure}

\begin{figure}[!t]
	\centering
 	\includegraphics[width=0.85\columnwidth]{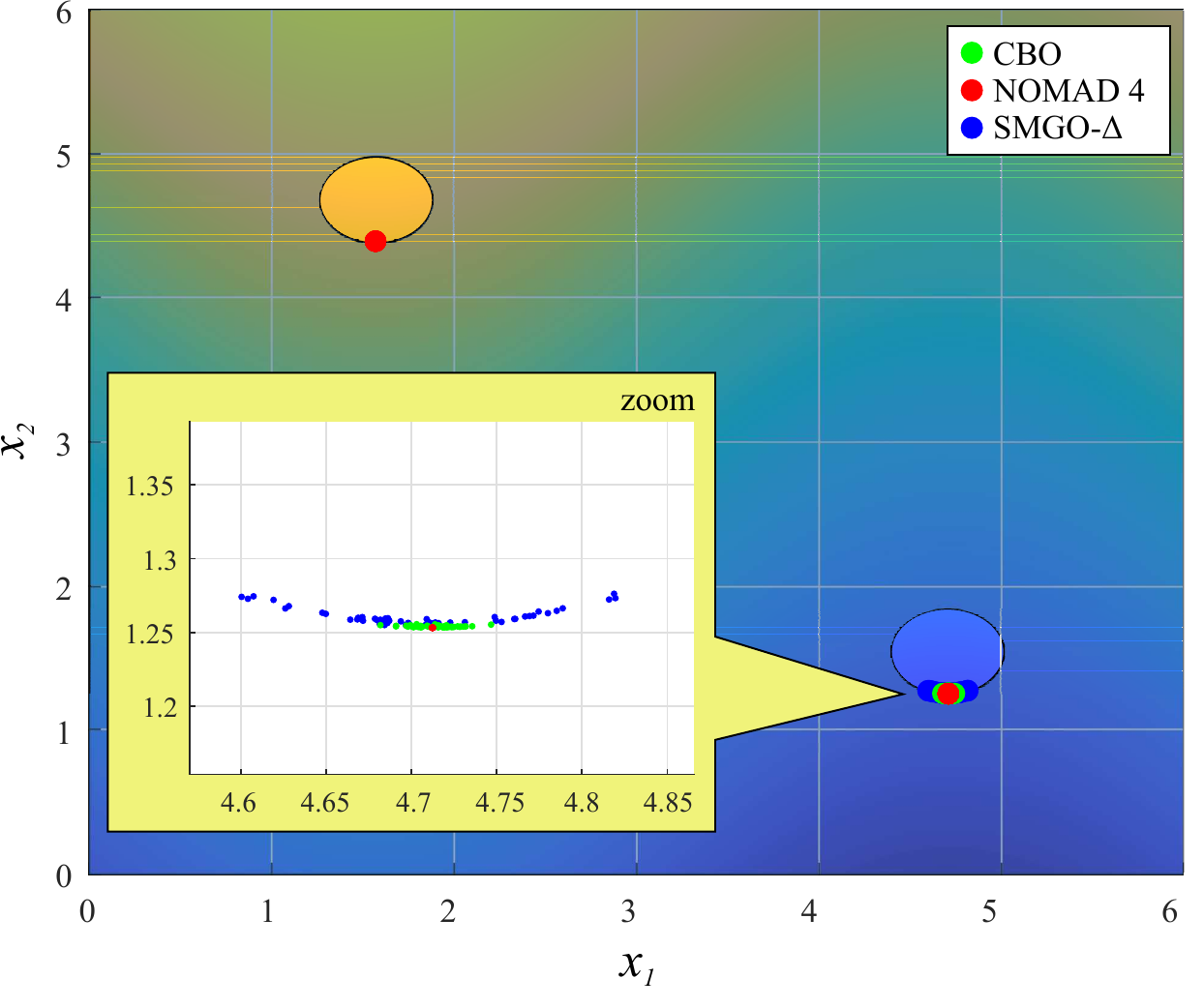}
	\caption{Benchmark tests: distribution of best results for \\ different optimization runs on T2.}
	\label{fig:t2}
\end{figure}

\begin{figure}[!t]
	\centering
 	\includegraphics[width=0.85\columnwidth]{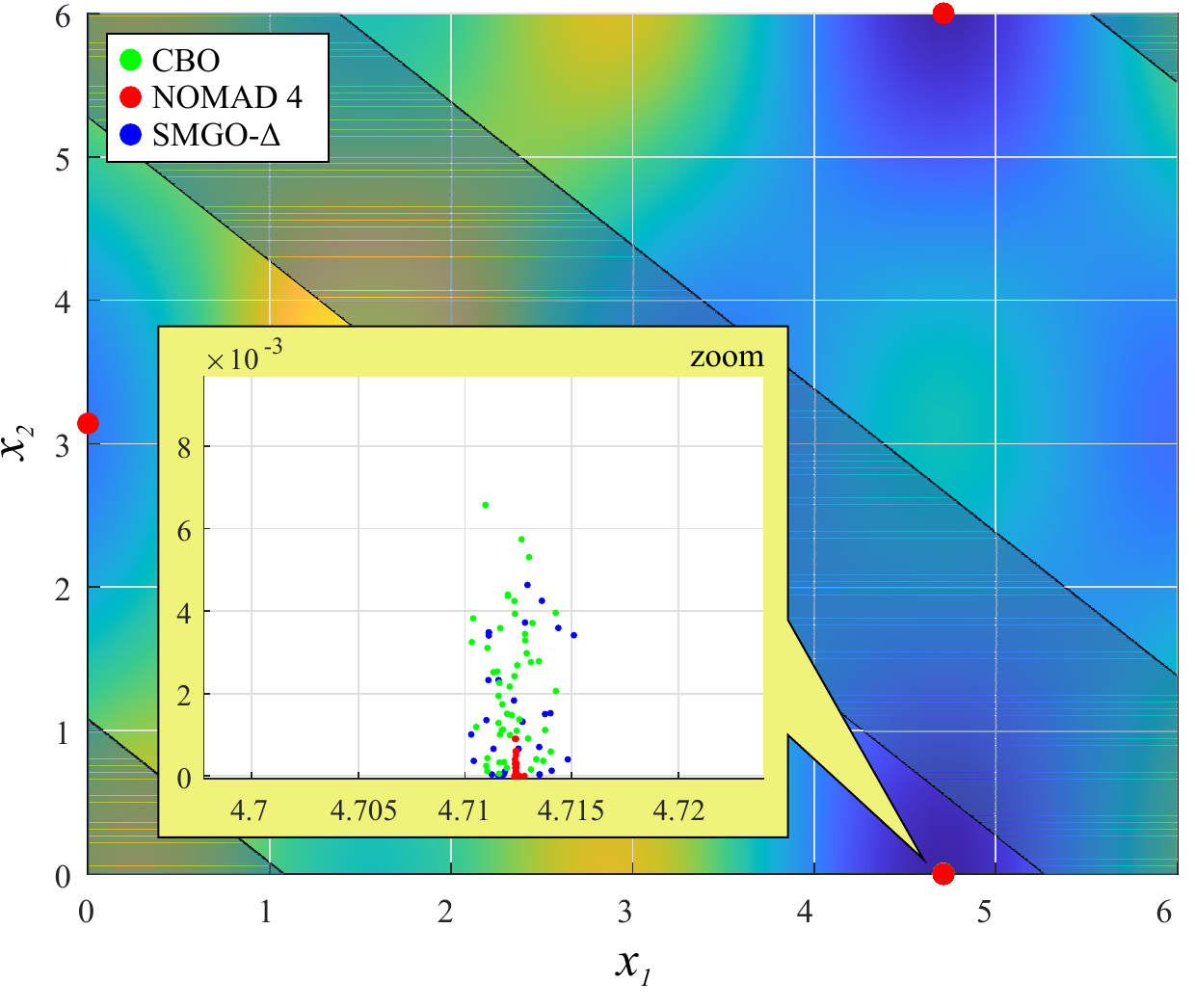}
	\caption{Benchmark tests: distribution of best results for \\ different optimization runs on T3.}
	\label{fig:t3}
\end{figure}

\subsection{Comparative results}

We summarize the average results across 50 runs of the respective compared algorithms in Table~\ref{table:syn-tests-result1}. In problems G04, G05MOD, and G12, we see that all three methods had highly comparable results, hence we did not highlight any best result in their corresponding rows. For most problems, CBO displayed the best results, for which it resulted in joint best average in 6 problems. For the others, it had comparable results except for G23MOD where CBO resulted in significantly worse average. NOMAD~4 had 3 joint best averages, and for problems T1-T3, it prematurely finished optimization because of maximum gridding, at around $n=200$ instead of the maximum $N=500$. This has led to a worst result for T2, which was 8 times larger than the results for CBO and SMGO-$\Delta$. SMGO-$\Delta$ produced competitive results compared with the other algorithms. In 5 problems, it was joint best with either CBO or NOMAD~4. However, its average results with G09 was around 2 times worse than the others, which can be attributed to the problem having relatively higher dimensions, combined with low $\rho$.

\begin{table}
\small
\begin{center}
	\begin{tabular}{|c|c|c|c|}
		\hline
		\textbf{Problem} & \textbf{CBO} & \textbf{NOMAD 4} & \textbf{SMGO-$\Delta$}\\ \hline
		G04    & -3.0518e+04 & -3.0665e+04 & -3.0343e+04\\ \hline
		G05MOD & 5.2185e+03  & 5.2073e+03  & 5.4014e+03\\ \hline
		G08    & \cye-0.0958     & -0.0823     & \cye-0.0958\\ \hline
		G09    & \cye754.1200    & \cye717.6586    & 1.5131e+03 \\ \hline
		G12    & -0.9998     & -1.000      & -0.9671 \\ \hline
		G23MOD & -3.3832e+03 & \cye-3.9000e+03 & \cye-3.9941e+03 \\ \hline
		G24    & \cye-5.4262     & \cye-5.3852     & -5.2789 \\ \hline
		T1     & \cye0.6005      & 0.6769$^*$      & \cye0.6088 \\ \hline
		T2     & \cye0.2542      & 1.9671$^{*,**}$      & \cye0.2628 \\ \hline
		T3     & \cye-2.0000     & -1.8904$^*$     & \cye-2.0000 \\ \hline
	\end{tabular}
	\caption{Benchmark tests: average results after 500 iterations. Best results are highlighted in yellow. $^*$: prematurely finished optimization because of maximum gridding refinement, $^{**}$: stopped with no results in 5 separate runs because of no feasible point found after around $n=90$.}  
	\label{table:syn-tests-result1}
\end{center}
\end{table}

Table~\ref{table:syn-tests-result2} shows the average number of iterations at the first feasible sample for the respective problems. CBO achieved the most sole/joint best results across all problems, implying its effectiveness in terms of finding a feasible point. SMGO-$\Delta$ had the best result for G23MOD, and joint best with CBO on three other problems. However, G05MOD was a particularly difficult problem for SMGO-$\Delta$ to find a feasible point, resulting in more than twice iterations as NOMAD~4. This can be associated with a fairly high number of constraints, for which different region/partitions in the search space have different combinations of (predicted) satisfied constraints. This means that in effect, a large region in the search space could have the same number of constraints satisfied, albeit in different combinations. By virtue of the exploration, which prioritizes regions based on number of constraints satisfied, it has taken many iterations of exploration before actually sampling a point that satisfies all 5 constraints. Lastly, NOMAD~4 had no best results in this metric, and with a significant worst result with G23MOD, taking 5 times longer than CBO and 20 times longer than SMGO-$\Delta$ before taking a first feasible sample. Furthermore, we point out that although NOMAD~4 had competitive results with problem T2, there were 5 runs in which NOMAD~4 terminated without any result because there was still no feasible point found after around $n=90$.

We have statistically compared the results of SMGO-$\Delta$ with the other algorithms using Wilcoxon and Kruskal-Wallis non-parametric tests, both with 5\% significance levels and a null hypothesis that SMGO-$\Delta$ results are similar to those from the other methods. We found out that SMGO-$\Delta$ distributions for the final results after $N=500$, as well as for the iterations for first feasible sampling, are not statistically similar to those of competitor techniques.

\begin{table}
\small
\begin{center}
	\begin{tabular}{|c|c|c|c|}
		\hline
		\textbf{Problem} & \textbf{CBO} & \textbf{NOMAD 4} & \textbf{SMGO-$\Delta$}\\ \hline
		G04    & \cye4.250  & 19.531 & \cye4.938\\ \hline
		G05MOD & \cye7.760  & 66.156 & 166.540\\ \hline
		G08    & \cye6.440  & 18.880 & 27.860\\ \hline
		G09    & \cye13.820 & 62.160 & 42.020 \\ \hline
		G12    & \cye13.000 & 29.087  & 25.500 \\ \hline
		G23MOD & 8.735  & 46.837 & \cye2.449 \\ \hline
		G24    & \cye3.121   & 6.485  & \cye2.667 \\ \hline
		T1     & \cye3.192  & 9.192  & \cye3.192 \\ \hline
		T2     & \cye8.694  & 20.227$^{**}$  & 24.102 \\ \hline
		T3     & \cye2.667   & 4.800  & 6.133 \\ \hline
	\end{tabular}
	\caption{Benchmark tests: average number of iterations before finding the first feasible point. Only the runs starting from an unfeasible point are considered. Best results are highlighted in yellow. $^{**}$: stopped with no results in 5 separate runs because of no feasible point found after around $n=90$.}  
	\label{table:syn-tests-result2}
\end{center}
\end{table}

\section{Case Study: MPC Tuning for a DC Motor}
\label{sec:expt-tests}

We finally use SMGO-$\Delta$ to tune a model predictive controller for a DC motor with model uncertainty and task-level constraints. We compare its performance with an ideal optimal controller with complete information about the actual system model at hand, optimizing the task-level costs considering the entire task duration. Furthermore, we also make a comparison with CBO.

\begin{figure}[!t]
	\centering
	\includegraphics[width=0.5\columnwidth]{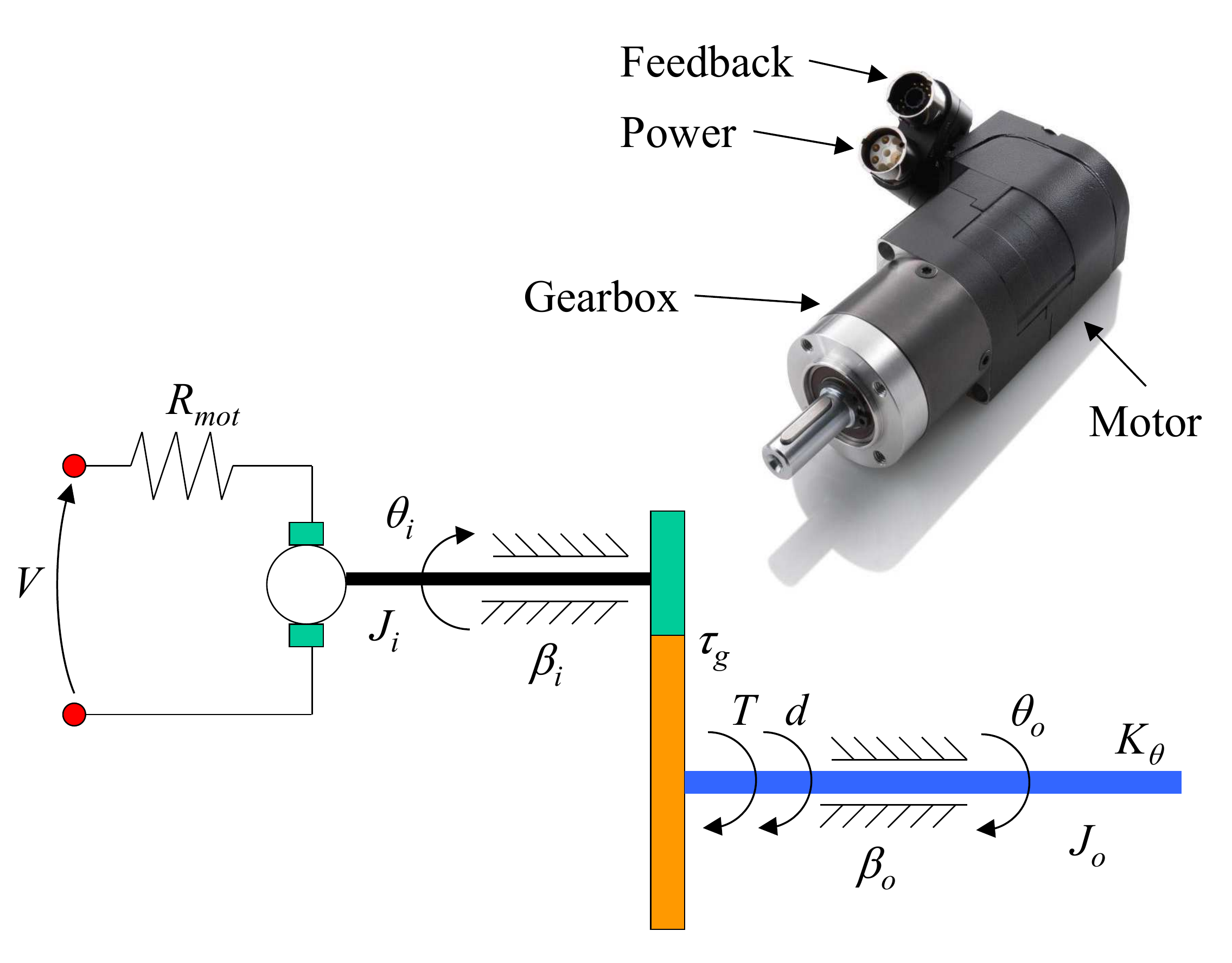}
	\caption{Case study: servomotor model considered \\ in the MPC tuning problem.}
	\label{fig:servo-model}
\end{figure}

\subsection{System Model and Controller Structure}
The servo positioning system is shown in Fig.~\ref{fig:servo-model}, where $V$ is the input voltage, $\theta_i$ and $\theta_o$ are the input and output shaft angular positions, $T$ is the torsional moment between the input and output shafts, and $d$ is the load (disturbance) on the output shaft. The nameplate nominal values for the motor parameters are given in Table~\ref{table:motor-specs}.

Assuming linearity of all the components, the system dynamics are described by the following state-space equations (the continuous time variable is omitted for brevity):

\begin{equation}
\label{eqn:servo-model}
\small
    \dot{\bm{\xi}}
    = 
    \begin{bmatrix}
    0 & 0 & 1 & 0 \\
    0 & 0 & 0 & 1 \\
    \frac{-K_\theta}{J_i \tau_g^2} & \frac{-K_\theta}{J_i \tau_g} &  -\frac{\beta_i+\frac{K_t^2}{R_{m}}}{J_i} & 0 \\
    \frac{K_\theta}{J_o \tau_g} & \frac{-K_\theta}{J_o} & 0 & \frac{-\beta_o}{J_o}
    \end{bmatrix}
    \bm{\xi}
    +
    \begin{bmatrix}
    0 \\ 0 \\ \frac{K_t}{J_i R_{m}} \\ 0
    \end{bmatrix}
    u
    +
    \begin{bmatrix}
    0 \\ 0 \\ 0 \\ \frac{1}{J_o}
    \end{bmatrix}
    d,
\end{equation}

\begin{equation}
\label{eqn:servo-out}
\small
    \bm{y} = \begin{bmatrix}
    \frac{K_\theta}{\tau_g} & -K_\theta & 0 & 0 \\
    0 & 1 & 0 & 0 \\
    0 & 0 & 1 & 0 \\
    0 & 0 & 0 & 1
    \end{bmatrix}
    \bm{\xi}
\end{equation}
\noindent where $\bm{\xi} = \lbrack \theta_i ~ \theta_o ~ \dot{\theta}_i ~ \dot{\theta}_o \rbrack^\top$ is the state, and $\bm{y} = \lbrack T ~ \theta_o ~ \dot{\theta}_i ~ \dot{\theta}_o \rbrack^\top$ the output.

For the considered system, we design a model predictive controller (MPC, \cite{borrelli_bemporad_morari_2017}) to track a reference output shaft angle $\hat{\theta}_o$. An MPC strategy is chosen in this case study to explicitly consider the input and state constraints, pertaining to the voltage input and the maximum torsional moments experienced by the shafts.

After setting a sampling time $T_s=0.1\,$s, at each time step $t$ the following finite horizon optimal control problem is solved:

\begin{equation}
\small
    \label{eqn:mpc-opt}
    \min_{U \in \mathbb{R}^{N_h}} \sum_{i=0}^{N_h} (y_{ref} - y(i|t))^\top \bm{Q} (y_{ref} - y(i|t)) + \sum_{i=0}^{N_h-1} R u(i|t)^2
\end{equation}
\begin{align}
\small
    \label{eqn:mpc-model1}
    \mathrm{s.t.} ~&~ \bm{\xi}(i+1|t) = \bm{A\xi}(i|t) + Bu(i|t) + B_d d(i|t) \\
    \label{eqn:mpc-model2}
    ~ &~ y(i|t) = \bm{Cx}(i|t) \\
    \label{eqn:mpc-model3}
    ~ &~ \bm{\xi}(0|t) = \bm{\xi}(t) \\
    \label{eqn:mpc-constr1}
    ~ &~ \big|u(i|t)\big| \leq \overline{V} \\
    \label{eqn:mpc-constr2}
    ~ &~
    \Big|T(i|t)\Big| \leq \overline{T},
\end{align}
\noindent where for each signal, the time instant $(i|t)$ indicates the value of the signal at time $t+i$ predicted at time $t$, (\ref{eqn:mpc-model1}) is the discrete-time equivalent of \eqref{eqn:servo-model} considering a sampling time of 0.1~s, \eqref{eqn:mpc-model2} is the output equation given in \eqref{eqn:servo-out}, \eqref{eqn:mpc-model3} is the known system state at time step $t$, (\ref{eqn:mpc-constr1}) declares the input constraint with $\overline{V} = 220$~V, and (\ref{eqn:mpc-constr2}) limits the torsional moment between the gears, with $\overline{T}=78$~Nm. $N_h$ is the prediction and control horizon, selected as $N_h=50$. Matrix $\bm{Q}$ is a diagonal positive-definite matrix of tunable parameters $q_1, q_2, q_3, q_4$, while $R$ is fixed as 1.

\begin{table}
\small
\begin{center}
	\begin{tabular}{|c|p{1.75in}|x{1in}|}
		\hline
		\textbf{Variable} & \textbf{Description} & \textbf{Nominal value} \\ \hline
		$R_m$ & Motor electrical resistance & 20~$\Omega$ \\ \hline
		$K_t$ & Motor constant & 10~$\frac{\text{Nm}}{\text{A}}$ \\ \hline
		$K_\theta$ & Output shaft torsional stiffness & 1280~$\frac{\text{Nm}}{\text{rad}}$ \\ \hline
		$J_i$ & Input shaft moment of inertia & 0.5~kg~m$^2$ \\ \hline
		$J_o$ & Output shaft moment of inertia & 25~kg~m$^2$ \\ \hline
		$\beta_i$ & Input shaft friction coefficient & 0.1~$\frac{\text{Nm~s}}{\text{rad}}$ \\ \hline
		$\beta_o$ & Output shaft friction coefficient & 25~$\frac{\text{Nm~s}}{\text{rad}}$ \\ \hline
		$\tau_g$ & Gear ratio (input/output) & 20 \\ \hline
	\end{tabular}
	\caption{Case study: servo motor specifications}  
	\label{table:motor-specs}
\end{center}
\end{table}

\begin{figure}[!t]
	\centering
	\includegraphics[width=0.65\columnwidth]{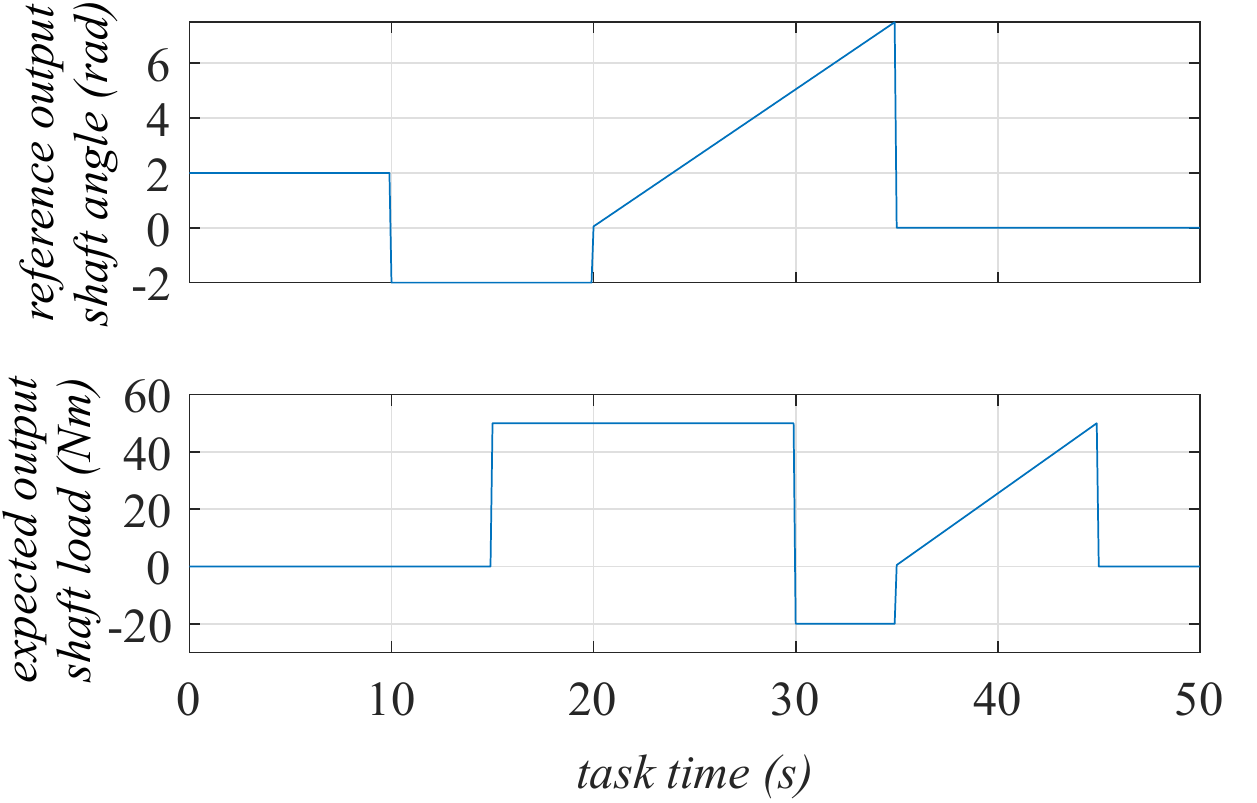}
	\caption{Case study: reference trajectory and expected load\\ for the servomotor problem}
	\label{fig:theta-ref-load}
\end{figure}

\subsection{Black-box optimization for controller tuning}
In the context of factory automation, we encounter machines performing repetitive tasks that usually involve a fixed reference trajectory $\hat{\theta}_o$ and load profile $d$. Fig.~\ref{fig:theta-ref-load} shows the (idealized) task cycle profile we consider, having a duration of $t_{cyc} = 50$~s. 

Once the controller architecture is defined, we aim to tune the MPC parameters to guarantee good tracking performance during the complete task, considering the following practically-motivated issues: 
\begin{enumerate}
    \item MPC cost function matrix $\bm{Q}$ needs to be tuned to maximize tracking performance on $\hat{\theta}_o$. This concern can usually be treated in the design phase without using black box optimization.
    \item The nameplate motor parameters are known (see Table~\ref{table:motor-specs}); however, the real ones differ due to manufacturing tolerances and/or previous usage. In this paper, $R_m, K_t, \beta_i, \beta_o$ may differ up to $\pm12.5\%$ w.r.t. nameplate values (the remaining parameters are assumed exact). This introduces model mismatch, possibly leading to degraded MPC performance, and/or violation of constraints (especially (\ref{eqn:mpc-constr2})). Hence, simultaneous model learning and MPC gain tuning is an appropriate approach for this application.
    \item To protect the motor windings and prolong its service lifetime, the average power consumption of the system throughout the task cycle, must not exceed a maximum $\overline{P}$, set as 25~W. At the same time, constraints \eqref{eqn:mpc-constr1} and \eqref{eqn:mpc-constr2} must be satisfied for the complete task duration. This declares a constraint affecting the whole task duration, i.e. longer than the MPC prediction horizon. Hence, the satisfaction of this constraint cannot be imposed \textit{a priori} in the MPC design phase.
    \item To minimize the cumulative damage to the input-output shaft gear coupling, we limit the total time that the torsional moment $T$ exceeds the limit $\overline{T}$, during the task duration, to be at most $\overline{t}_{viol}$, with $\overline{t}_{viol}=1$~s.
\end{enumerate}

Based on the previous considerations, the following black-box optimization problem is formulated:

\begin{equation}
\label{eqn:motor-tuning-cost}
    \min_{\bm{x} \in \mathcal{X}} \sum_{t=0}^{t_{max}} (\theta_o(t;\bm{x}) - \hat{\theta}_o(t))^2
\end{equation}

\begin{equation}
\label{eqn:max-tormoment}
    \mathrm{s.t.} ~\int_{t=0}^{t_{cyc}}\left(
    \begin{cases}
      1 & \big| T(t) \big| > \overline{T}\\
      0 & \text{otherwise}
    \end{cases}
   \right)dt \leq \overline{t}_{viol},
\end{equation}
    
\begin{equation}
\label{eqn:max-power}
\small
       \frac{1}{t_{cyc}+1}
    \sum_{t=0}^{t_{cyc}-1} \frac{u(t)\left(u(t) - K_t \dot{\theta}_i(t) \right)}{R_m} \leq \overline{P}
\end{equation}

\noindent where the vector of decision variables is $\bm{x} = \lbrack log(q_1) ~ log(q_2) ~ log(q_3) ~ log(q_4) ~ \tilde{R}_m ~ \tilde{K}_t ~ \tilde{\beta}_i ~ \tilde{\beta}_o \rbrack$. The tilde on 4th to 8th variables of $\bm{x}$ denotes parameter estimates. $\theta_o(t;\bm{x})$ is the resulting output shaft angular position at time $t$, when the MPC controller is applied to the servo system using the parameters $\bm{x}$. Constraint \eqref{eqn:max-tormoment} sets the maximum time of torsional moment violation for the experiment, while \eqref{eqn:max-power} imposes the maximum average motor power for the entire task duration.


The search intervals for the first four parameters is $[-7,\,~7]$, while intervals of $\pm 12.5\%$ w.r.t. nominal values are imposed for the next four (see Table~\ref{table:motor-specs} for the nominal values). 25 trials of the SMGO-$\Delta$ algorithm are executed with $N=250$ iterations each, with randomly-generated values of $R_m, K_t, \beta_i, \beta_o$ within the tolerances around the nominal values. For comparison, we also run the constrained Bayesian optimization (CBO, available as \verb|bayesopt| in MATLAB Statistics and Machine Learning Toolbox). 

Each trial has the same values of $R_m, K_t, \beta_i, \beta_o$ for both optimizers to guarantee uniformity in the tests. Furthermore, the starting test point $\bm{x}^{(1)}$ for each trial (both SMGO-$\Delta$ and CBO) is composed of $log(q_i) = 0$, and the nominal values for each motor parameter. Each experiment (iteration) is subject to the shaft load as in Fig.~\ref{fig:theta-ref-load}, but with a different disturbance (2.5~Nm amplitude white noise, band limited to 5~Hz). Again, to maximize uniformity, each iteration has exactly \textit{the same} shaft load disturbance for SMGO-$\Delta$ and CBO. That is, the experiments differ from iteration to iteration (due to the load disturbance), but are the same across the tested optimizers. 



\subsection{Comparative Results}

A comparison of the results from SMGO-$\Delta$ and CBO is laid out in Fig.~\ref{fig:results-mpc-tuning}. The results show the distributions of best samples history (first row), the corresponding constraint values for the best sample (second and third rows), and the count of unfeasible samples within the last 20 iterations (last row). The individual values from the 25 independent trials, as well as the mean values, are shown for both algorithms. 

\begin{figure}[!t]
	\centering
 	\includegraphics[width=\columnwidth]{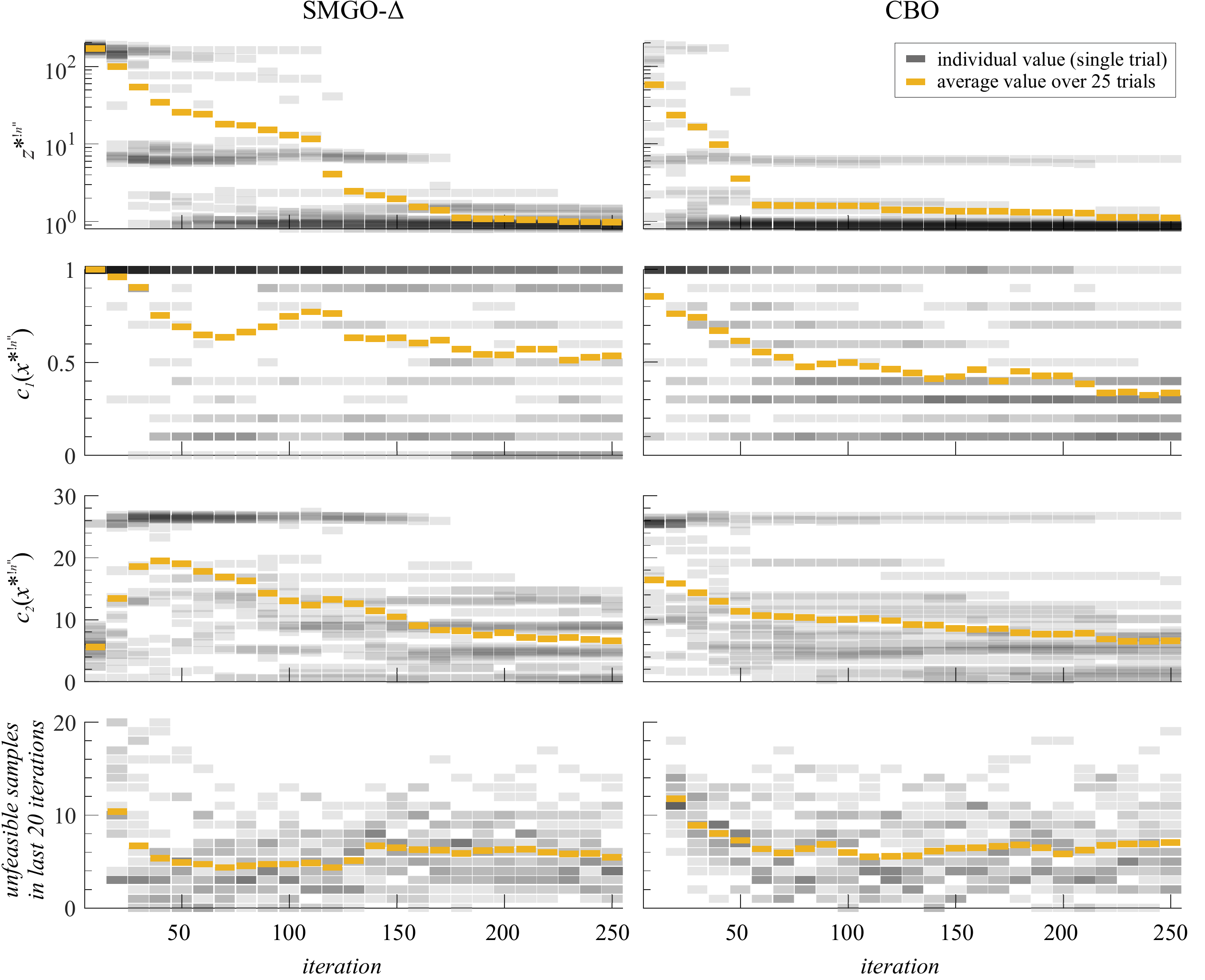}
	\caption{Case study: MPC tuning results comparison, SMGO-$\Delta$ vs CBO: best objective, \\constraints, and count of unfeasible samples in the previous 20 iterations}
	\label{fig:results-mpc-tuning}
\end{figure}

The iteration-based optimization performance of SMGO-$\Delta$ was slower than CBO for this particular test problem, particularly around iterations 30 to 110, where SMGO-$\Delta$ suffered from slow improvement of best feasible point, while CBO was already in the range of 1.10 (average) best objective. However, from around $n=170$, SMGO-$\Delta$ managed to produce better average results than CBO, and takes this advantage until the allotted 250 iterations. On the other hand, there were several trials for CBO whose best results were stuck at around 7.0 even after $n=200$. At the end of the iteration budget, we have obtained an average best objective of 0.979 for SMGO-$\Delta$, and 1.105 for CBO.

The distribution histories for the constraints (for the best sample $\bm{x}^{*\iter{n}}$) show that the best sample is found more and more towards the boundary of feasible set $\mathcal{G}$. This is a reasonable trend, because demanding better tracking performance also means leaning on more aggressive controllers, which in turn are more likely to violate maximum torsional moment and power limits. 

The count of unfeasible samples are lower with SMGO-$\Delta$ than with CBO, as seen in the last row of Fig.~\ref{fig:results-mpc-tuning}. We see that the slow improvement on SMGO-$\Delta$ around $n=30$ to $n=110$ coincide with the low count of unfeasible samples, which can be because the higher tendency of exploitation at these iterations. However, on the later part of the runs, the average unfeasible samples count for SMGO-$\Delta$ and CBO are similar. After the allotted 250 iterations, the average total unfeasible samples of SMGO-$\Delta$ (29.6\%) are less than that of CBO (34.7\%). 

A clear advantage can be seen when comparing the computational times of the two algorithms (not including the long function evaluation times), shown in Fig.~\ref{fig:results-comp-times}. It is clear that the proposed SMGO-$\Delta$ is much faster than Bayesian, with per-iteration computation times almost 2 orders of magnitude faster. The (average) total time for SMGO-$\Delta$ in the entire optimization run (250 iterations) is found to be 20.1~s, which is 51 times shorter than CBO (with 1030~s). Given the much lighter computational burden, and a competitive iteration-based optimization performance, SMGO-$\Delta$ can be argued to be a worthy alternative to CBO.

\begin{figure}[!t]
	\centering
 	\includegraphics[width=0.65\columnwidth]{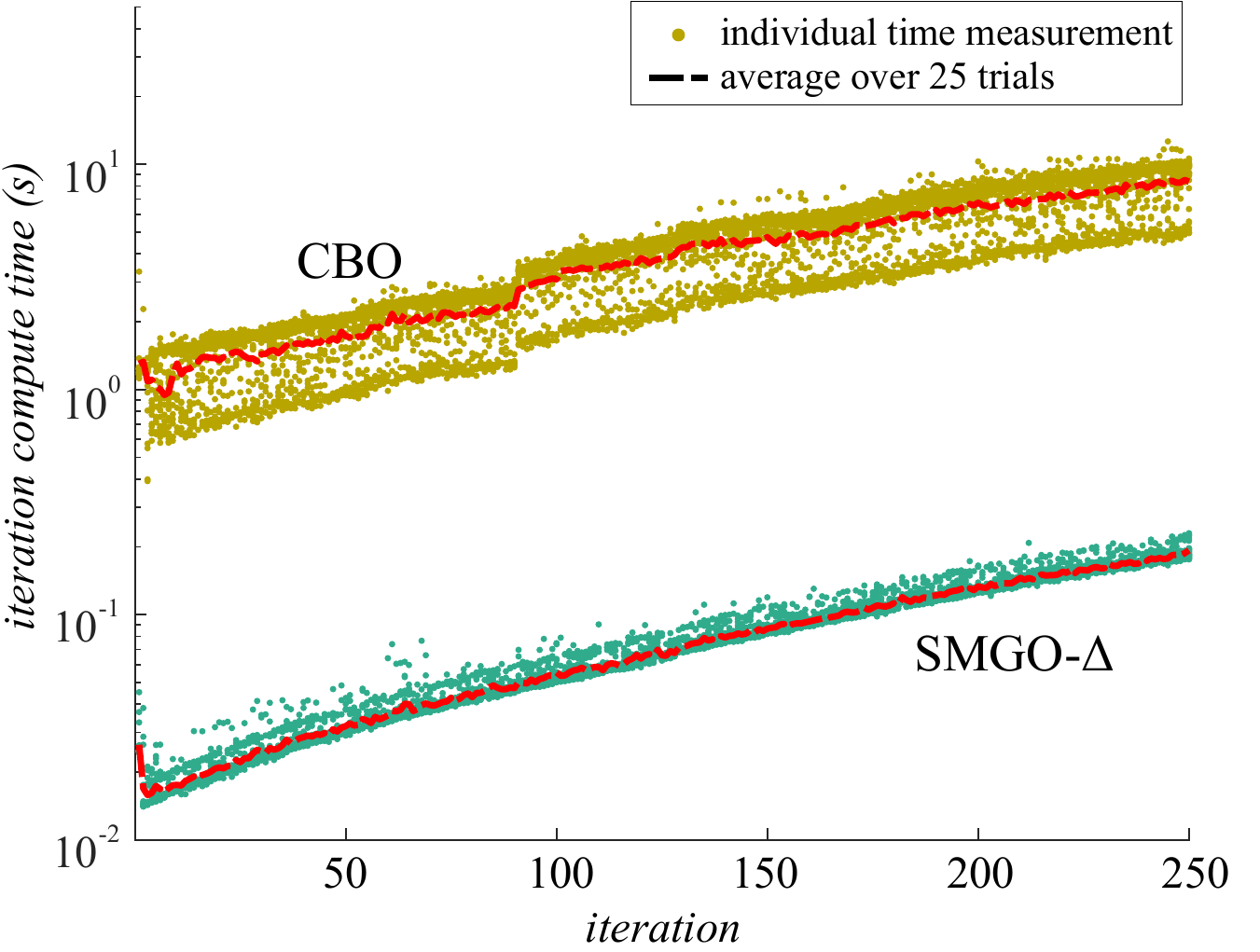}
	\caption{Case study: computational times comparison, SMGO-$\Delta$ vs CBO}
	\label{fig:results-comp-times}
\end{figure}


\section{Conclusions}
A global optimization algorithm based on the Set Membership framework, named SMGO-$\Delta$, is designed for cases where the objective as well as the constraints are black-box, e.g. both can only be evaluated through sampling, possibly affected by bounded noise. The exploitation and exploration strategies of SMGO-$\Delta$ are formulated based on the calculation of guaranteed bounds on the cost function and constraints, according to the Set Membership theory. The convergence properties of the proposed algorithm are proven in a theoretical analysis under a Lipschitz continuity assumption of the cost and constraint functions. Furthermore, extensions to noisy samples, enrichment of search directions, and computational cost are discussed. The designed algorithm features good iteration-based performance, repeatability, low computational complexity, and intuitive tunable risk parameter.

The performance and the sensitivity of the algorithm to its user-defined parameters are investigated with an illustrative example function. The SMGO-$\Delta$ is compared with state-of-the-art methods in academic benchmarks, showing competitive iteration-based performance. Furthermore, SMGO-$\Delta$ has been compared side-by-side with constrained Bayesian optimization in the engineering design context of MPC tuning for an electromechanical system with uncertain parameters. The tests show the highly comparable performance of SMGO-$\Delta$, with the additional advantage of being around 50 times faster than constrained Bayesian optimization for the treated test case.
\label{sec:conclusion}

\appendix

\section*{Appendix. Problem Definitions for Benchmark Tests}

\subsection*{G04 \cite{Liang2006}}
\begin{table}[H]
\small
\begin{center}
	\begin{tabular}{|c|p{4.5in}|}
	\hline
	\textbf{Description} & \textbf{Value} \\ \hline
	Objective & $f(\bm{x}) = 5.3578547x_3^2 + 0.8356891x_1x_5 + 37.293239x_1 - 40792.141$ \\ \hline
	\multirow{6}{*}{Constraints} & $g_1(\bm{x}) = 85.334407 + 0.0056858x_2x_5 + 0.0006262x_1x_4 - 0.0022053x_3x_5 - 92 \leq 0$ \\ 
	& $g_2(\bm{x}) = -85.334407 - 0.0056858x_2x_5 - 0.0006262x_1x_4 + 0.0022053x_3x_5 \leq 0$ \\ 
	& $g_3(\bm{x}) = 80.51249 + 0.0071317x_2x_5 + 0.0029955x_1x_2 + 0.0021813x^2_3 - 110 \leq 0$ \\ 
	& $g_4(\bm{x}) = -80.51249 - 0.0071317x_2x_5 - 0.0029955x_1x_2 - 0.0021813x^2_3 + 90 \leq 0$ \\ 
	& $g_5(\bm{x}) = 9.300961 + 0.0047026x_3x_5 + 0.0012547x_1x_3 + 0.0019085x_3x_4 - 25 \leq 0$ \\ 
	& $g_6(\bm{x}) = -9.300961 - 0.0047026x_3x_5 - 0.0012547x_1x_3 - 0.0019085x_3x_4 + 20 \leq 0$ \\ \hline
	\multirow{3}{*}{Search space} & $78 \leq x_1 \leq 102,$ \\
	& $33 \leq x_2 \leq 45,$ \\
	& $27 \leq x_i \leq 45, i=3,4,5$ \\ \hline
	\end{tabular}
	\label{table:g04-defn}
\end{center}
\end{table}

\subsection*{G05MOD \cite{Jiang2021,Liang2006}}
\begin{table}[H]
\small
\begin{center}
	\begin{tabular}{|c|p{4in}|}
	\hline
	\textbf{Description} & \textbf{Value} \\ \hline
	Objective & $f(\bm{x}) = 3x_1 + 0.000001x^3_1 + 2x_2 + (0.000002/3)x^3_2$ \\ \hline
	\multirow{5}{*}{Constraints} & $g_1(\bm{x}) = x_3 - x_4 - 0.55 \leq 0$ \\ 
	& $g_2(\bm{x}) = x_4 - x_3 - 0.55 \leq 0$ \\ 
	& $g_3(\bm{x}) = 1000 \mathrm{sin}(-x_3 - 0.25) + 1000 \mathrm{sin}(-x_4 - 0.25) + 894.8 - x_1 \leq 0$ \\ 
	& $g_4(\bm{x}) = 1000 \mathrm{sin}(x_3 - 0.25) + 1000 \mathrm{sin}(x_3 - x_4 - 0.25) + 894.8 - x_2 \leq 0$ \\ 
	& $g_5(\bm{x}) = 1000 \mathrm{sin}(x_4 - 0.25) + 1000 \mathrm{sin}(x_4 - x_3 - 0.25) + 1294.8 \leq 0$ \\ \hline
	\multirow{2}{*}{Search space} & $0 \leq x_i \leq 1200, i=1,2$ \\
	& $-0.55 \leq x_j \leq 0.55, j=3,4$ \\ \hline
	\end{tabular}
	\label{table:g05mod-defn}
\end{center}
\end{table}

\subsection*{G08 \cite{Liang2006}}
\begin{table}[H]
\small
\begin{center}
	\begin{tabular}{|c|p{2in}|}
	\hline
	\textbf{Description} & \textbf{Value} \\ \hline
	Objective & $f(\bm{x}) = -\frac{\mathrm{sin}^3(2\pi x_1) \mathrm{sin}(2\pi x_2)}{x_1^3(x_1 + x_2)}$ \\ \hline
	\multirow{2}{*}{Constraints} & $g_1(\bm{x}) = x_1^2 - x_2 + 1 \leq 0$ \\ 
	& $g_2(\bm{x}) = 1 - x_1 + (x_2 - 4)^2 \leq 0$ \\ \hline
	Search space & $0 \leq x_i \leq 10, i=1,2$ \\ \hline
	\end{tabular}
	\label{table:g08-defn}
\end{center}
\end{table}

\subsection*{G09 \cite{Liang2006}}
\begin{table}[H]
\small
\begin{center}
	\begin{tabular}{|c|p{4.5in}|}
	\hline
	\textbf{Description} & \textbf{Value} \\ \hline
	Objective & $f(\bm{x}) = (x_1 - 10)^2 + 5(x_2 - 12)^2 + x^4_3 + 3(x_4 - 11)^2 +10x^6_5 + 7x^2_6 + x^4_7 - 4x_6x_7 - 10x_6 - 8x_7$ \\ \hline
	\multirow{4}{*}{Constraints} & $g_1(\bm{x}) = -127 + 2x^2_1 + 3x^4_2 + x_3 + 4x^2_4 + 5x_5 \leq 0$ \\ 
	& $g_2(\bm{x}) = -282 + 7x_1 + 3x_2 + 10x^2_3 + x_4 - x_5 \leq 0$ \\ 
	& $g_3(\bm{x}) = -196 + 23x_1 + x^2_2 + 6x^2_6 - 8x_7 \leq 0$ \\ 
	& $g_4(\bm{x}) = 4x^2_1 + x^2_2 - 3x_1x_2 + 2x^2_3 + 5x_6 - 11x_7 \leq 0$ \\ \hline
	Search space & $-10 \leq x_i \leq 10, i=1,\ldots,7$ \\ \hline
	\end{tabular}
	\label{table:g09-defn}
\end{center}
\end{table}

\subsection*{G12 \cite{Liang2006}}
\begin{table}[H]
\small
\begin{center}
	\begin{tabular}{|c|p{4.0in}|}
	\hline
	\textbf{Description} & \textbf{Value} \\ \hline
	Objective & $f(\bm{x}) = -\left(100 - (x_1 - 5)^2 - (x_2 - 5)^2 - (x_3 - 5)^2\right)/100$ \\ \hline
	Constraints & $g_1(\bm{x}) = (x_1 - p)^2 + (x_2 - q)^2 + (x_3 - r)^2 - 0.0625 \leq 0, ~ p,q,r = 1,\ldots,9$ \\ \hline
	Search space & $0 \leq x_i \leq 9, i=1,2,3$ \\ \hline
	Comments & The feasible region is a set of $9^3$ disjoint balls with radius of 0.25, laid out as a grid. Search space is modified so that $\bm{x}^*$ does not lie in the center. \\ \hline
	\end{tabular}
	\label{table:g12-defn}
\end{center}
\end{table}

\subsection*{G23MOD \cite{Liang2006}}
\begin{table}[H]
\small
\begin{center}
	\begin{tabular}{|c|p{4.0in}|}
	\hline
	\textbf{Description} & \textbf{Value} \\ \hline
	Objective & $f(\bm{x}) = -9x_5 - 15x_8 + 6x_1 + 16x_2 + 10(x_6 + x_7)$ \\ \hline
	\multirow{2}{*}{Constraints} & $g_1(\bm{x}) = x_9x_3 + 0.02x_6 - 0.025x_5 \leq 0$ \\ 
	& $g_2(\bm{x}) = x_9x_4 + 0.02x_7 - 0.015x_8 \leq 0$ \\ \hline
	\multirow{4}{*}{Search space} & $0 \leq x_1, x_2, x_6 \leq 300$ \\
	& $0 \leq x_3, x_5, x_7 \leq 100$ \\
	& $0 \leq x_4, x_8 \leq 200$ \\
	& $0.01 \leq x_9 \leq 0.03$ \\ \hline
	Comments & Taken from G23 as in \cite{Liang2006}, but removed the equality constraints \\ \hline
	\end{tabular}
	\label{table:g23mod-defn}
\end{center}
\end{table}

\subsection*{G24 \cite{Liang2006}}
\begin{table}[H]
\small
\begin{center}
	\begin{tabular}{|c|p{3.0in}|}
	\hline
	\textbf{Description} & \textbf{Value} \\ \hline
	Objective & $f(\bm{x}) = -x_1 - x_2$ \\ \hline
	\multirow{2}{*}{Constraints} & $g_1(\bm{x}) = -2x^4_1 + 8x^3_1 - 8x^2_1 + x_2 - 2 \leq 0$ \\ 
	& $g_2(\bm{x}) = -4x^4_1 + 32x^3_1 - 88x^2_1 + 96x_1 + x_2 - 36 \leq 0$ \\ \hline
	\multirow{2}{*}{Search space} & $0 \leq x_1 \leq 3$ \\
	& $0 \leq x_2 \leq 4$ \\ \hline
	\end{tabular}
	\label{table:g24-defn}
\end{center}
\end{table}

\subsection*{T1 \cite{Ariafar2019,Hernandez-Lobato2015,Hernandez-Lobato2016}}
\begin{table}[H]
\small
\begin{center}
	\begin{tabular}{|c|p{3in}|}
	\hline
	\textbf{Description} & \textbf{Value} \\ \hline
	Objective & $f(\bm{x}) = x_1 + x_2$ \\ \hline
	\multirow{2}{*}{Constraints} & $g_1(\bm{x}) = 0.5 \mathrm{sin}(2\pi(x^2_1 - 2x_2)) + x_1 + 2x_2 - 1.5 \geq 0$ \\ 
	& $g_2(\bm{x}) = -x^2_1 - x^2_2 + 1.5 \geq 0$ \\ \hline
	Search space & $0 \leq x_i \leq 1, i=1,2$ \\ \hline
	\end{tabular}
	\label{table:t1-defn}
\end{center}
\end{table}

\subsection*{T2 \cite{Ariafar2019,Gardner2014}}
\begin{table}[H]
\small
\begin{center}
	\begin{tabular}{|c|p{2in}|}
	\hline
	\textbf{Description} & \textbf{Value} \\ \hline
	Objective & $f(\bm{x}) = \mathrm{sin}(x_1) + x_2$ \\ \hline
	Constraints & $g_1(\bm{x}) = \mathrm{sin}(x_1) \mathrm{sin}(x_2) + 0.95 \leq 0$ \\ \hline
	Search space & $0 \leq x_i \leq 6, i=1,2$ \\ \hline
	\end{tabular}
	\label{table:t2-defn}
\end{center}
\end{table}

\subsection*{T3 \cite{Ariafar2019,Gardner2014}}
\begin{table}[H]
\small
\begin{center}
	\begin{tabular}{|c|p{3in}|}
	\hline
	\textbf{Description} & \textbf{Value} \\ \hline
	Objective & $f(\bm{x}) = \mathrm{cos}(2x_1) \mathrm{cos}(x_2) + \mathrm{sin}(x_1)$ \\ \hline
	Constraints & $g_1(\bm{x}) = \mathrm{cos}(x_1) \mathrm{cos}(x_2) - \mathrm{sin}(x_1) \mathrm{sin}(x_2) - 0.5 \leq 0$ \\ \hline
	Search space & $0 \leq x_i \leq 6, i=1,2$ \\ \hline
	\end{tabular}
	\label{table:t3-defn}
\end{center}
\end{table}

\bibliographystyle{IEEEtran}
\bibliography{smgo-d-paper}

\end{document}